\documentclass[11pt,reqno]{amsart}

\usepackage{amsrefs}
\usepackage{amsmath}
\usepackage{amsthm}
\usepackage{graphicx}
\usepackage{subcaption}
\usepackage{tikz}
\usetikzlibrary{trees,decorations,external,backgrounds,patterns}
\usepackage{amsfonts}
\usepackage{amssymb}
\usepackage{color}
\usepackage{bm}
\usepackage{hyperref}
\usepackage{enumitem}
\usepackage[margin=1.1in]{geometry}
\usepackage{relsize}
 \usepackage{floatrow}

\def\red{\textcolor{red}}

\newif\ifshownotes
\newcommand{\draftnote}[1]
{
  \ifshownotes
   \footnote{\red{#1}}
  \else
  \fi
}

\numberwithin{equation}{section}

\newtheorem{theorem}{Theorem}[section]
\newtheorem{lemma}[theorem]{Lemma}
\newtheorem{proposition}[theorem]{Proposition}

\newtheorem{corollary}[theorem]{Corollary}

\theoremstyle{definition}
\newtheorem{example}[theorem]{Example}
\newtheorem{definition}[theorem]{Definition}
\newtheorem{remark}[theorem]{Remark}

\newtheorem{assumption}{Assumption}

\newcommand{\R}{\stsets{R}}

\newcommand{\N}{\stsets{N}}

\def\E{{\mathbb E}}

\def\R{{\mathbb R}}

\def\N{{\mathbb N}}
\def\P{{\mathbb P}}

\def\Z{{\mathbb Z}}

\def\D{{\mathcal D}}

\newcommand{\Fmb}{{\mathbb{F}}}

\newcommand{\Qmb}{{\mathbb{Q}}}
\newcommand{\Rmb}{{\mathbb{R}}}

\newcommand{\Bmc}{{\mathcal{B}}}

\newcommand{\Dmc}{{\mathcal{D}}}
\newcommand{\Emc}{{\mathcal{E}}}
\newcommand{\Fmc}{{\mathcal{F}}}

\newcommand{\Hmc}{{\mathcal{H}}}

\newcommand{\Jmc}{{\mathcal{J}}}

\newcommand{\Lmc}{{\mathcal{L}}}
\newcommand{\Mmc}{{\mathcal{M}}}

\newcommand{\Pmc}{{\mathcal{P}}}

\newcommand{\Smc}{{\mathcal{S}}}
\newcommand{\Tmc}{{\mathcal{T}}}
\newcommand{\Umc}{{\mathcal{U}}}
\newcommand{\Vmc}{{\mathcal{V}}}
\newcommand{\Wmc}{{\mathcal{W}}}

\newcommand{\Ymc}{{\mathcal{Y}}}

\newcommand{\Xhat}{\hat{X}}

\newcommand{\gbar}{{\bar{g}}}

\newcommand{\Psibar}{{\bar{\Psi}}}

\newcommand{\G}{\mathcal{G}}

\newcommand{\Erdos}{Erd\H{o}s-R\'enyi}
\newcommand{\CM}{\textnormal{CM}}

\newcommand{\set}[1]{\left\{#1\right\}}

\newcommand{\Rd}{{\Rmb^d}}

\newcommand{\one}{{\boldsymbol{1}}}

\newcommand{\Xtil}{{\tilde{X}}}

\def\root{\varnothing}
\newcommand{\natzero}{\N}
\newcommand{\law}{\Lmc}
\DeclareSymbolFont{symbolsC}{U}{pxsyc}{m}{n}
\DeclareMathSymbol{\coloneqq}{\mathrel}{symbolsC}{"42}

\newcommand{\deq}{\,{\buildrel d \over =}\,}

\newcommand{\V}{\mathbb{V}}

\usepackage{mathtools}

\newcommand{\cl}{\text{cl}}
\newcommand*\diff{\mathop{}\!\mathrm{d}}

\usepackage{accents}

\DeclareMathOperator{\Supp}{Supp}

\newcommand{\stateS}{\mathcal{X}}
\newcommand{\X}{\Ymc}
\newcommand{\Nlf}{\mathbf{N}}

\newcommand{\tree}{\Tmc}
\newcommand{\Nmf}{\mathbf{N}}
\newcommand{\Xlf}{\Xtil}
\newcommand{\Xmf}{\Xhat}

\newcommand{\emp}{\mu}

\newcommand{\lawMF}{\mathfrak{p}}
\renewcommand{\r}{\rho}
\newcommand{\rlf}{\tilde{\r}}
\newcommand{\rmf}{\hat{\r}}

\newcommand{\rmp}{\gamma}

\newcommand{\extra}{\star}
\newcommand{\offspring}{\theta}

\newcommand{\dmax}{d_{\max}}
\newcommand{\ev}{\mathbf{e}^v}
\newcommand{\rr}{\mathbf{r}}

\newcommand{\countm}{\nu}

\newcommand{\YLE}{\Xlf}

\newcommand{\supoff}{\Theta}

\newcommand{\excit}{\alpha}

\newcommand{\tauone}{\bar{\tau}}

\newcommand{\Xlfinitial}{Y}
\newcommand{\Xmfinitial}{\Xlfinitial}

\renewcommand{\k}{K}

\newcommand{\empirdegree}{\bar{\offspring}}

\newcommand{\vertexS}{\mathcal{V}}

\newcommand{\Xmp}{\hat{Z}}
\newcommand{\Xjmp}{\widetilde{Z}}
\newcommand{\Xjmpinitial}{\Xjmp}
\newcommand{\ra}{\lambda}
\newcommand{\rjmp}{\widetilde{\ra}}
\newcommand{\rmjmp}{\hat{\ra}}
\newcommand{\Xcanon}{\bar{Z}}
\newcommand{\Gen}{\mathcal{A}}
\def\acknowledgementsname{Acknowledgments}
\newenvironment{acks}[1][\acknowledgementsname]{\section*{#1}}{\par}
 
\if@balayout
    \renewenvironment{acks}[1][\acknowledgementsname]%
        {%
            \vskip0.5\baselineskip
            \small
            {\noindent\normalfont\sffamily\bfseries\acknowledgementsname}\par
            \begingroup\parindent 0pt\parskip 0.5\baselineskip
        }%
        {\endgroup}
\fi
\def\fundingname{Funding}
\newenvironment{funding}[1][\fundingname]{\begin{acks}[\fundingname]}{\end{acks}}

\title[Tractable description of hydrodynamic limits of IPS on sparse graphs]{Tractable description of hydrodynamic limits of a class of interacting jump processes on sparse graphs}

\author[Cocomello]{Juniper Cocomello}     \author[Davydov]{ Michel Davydov}     
\author[Ramanan]{Kavita Ramanan}

\address{Division of Applied Mathematics, Brown University, 182 George Street, Providence, RI 02912} 
\email{juniper\_cocomello@brown.edu, michel\_davydov@brown.edu,kavita\_ramanan@brown.edu, }



\begin{document}
\begin{abstract} 
We  consider dynamics of the  empirical measure of vertex neighborhood states of 
 Markov interacting jump processes on  sparse random graphs, in a  suitable asymptotic limit as the graph size goes to infinity. 
  Under the assumption of a  certain acyclic structure on single-particle transitions, we provide a tractable autonomous description of the evolution of 
  this hydrodynamic limit in terms of a finite coupled system
  of ordinary differential equations. 
  Key ingredients of the proof include a  characterization of the hydrodynamic limit of the neighborhood empirical measure
  in terms of a certain local-field equation,  well-posedness of its Markovian projection, and a Markov random field property of the time-marginals,
  which may be of independent interest.    We also show how our results lead to   principled approximations for classes of interacting jump
  processes  and illustrate its efficacy via simulations on several examples, including an idealized model of seizure spread in the brain.     
\end{abstract}

\maketitle

\section{Introduction} 
\label{sec:intro}
\subsection{Interacting particle systems on graphs} 
\label{subsec:intro_IPS_on_graphs}
We consider large collections of randomly evolving interacting pure jump processes whose transition rates depend only on their own state and that of the states of neighboring particles with respect to an underlying interaction graph.
Such particle systems model a wide range of phenomena, including neuronal spiking models in the brain where particles represent neurons and the
interaction graph is a brain network \cites{truccolo_2005,Fournier_Locherbach_2016},  epidemiological compartmental models,
where individuals become infected depending on the states of their neighbors in a social contact network \cite{Pastor_Satorras_2015}, 
models of opinion dynamics
\cites{FraiLinOlvera24, AndOlvera24, baldassarri2024opiniond, RemcoShneer23}, or
load balancing in computer networks, where particles represent queues and the interaction graph is determined by the routing structure \cites{VveDobKar96,Mit01,AghRam17}. 
Key quantities of interest include macroscopic averages of functionals of the state such as the fractions of neighboring vertices in a given pair of states at a given time, or functionals of the histories such as the 
fractions of vertices that have no transitions in a time interval.  
These can be captured by the dynamics of the empirical measure  of the states of a vertex and the vertices in its neighborhood, or the path empirical measure of the process, respectively.  
These are complex high-dimensional processes that are in general not amenable to exact analysis.
A natural alternative is then to find tractable approximations that can be rigorously shown to be accurate in a relevant asymptotic limit. 

Our approach is based on the framework of local-field limits, which was introduced in the context of interacting diffusions in \cites{lacker2023local,lacker2023marginal}, and developed for interacting continuous time Markov chains in \cites{GanRam2024,Gangulythesis,GanRamTree,GanRamUGW}. 
Specifically, it was shown in \cite{GanRam2024}*{Corollary 4.12}  that for a very broad class of interacting jump processes taking values in a discrete state space, given a sequence of  graphs $G_n$ that converges locally to a limit graph $G$, the corresponding sequence of path empirical measures on $G_n$ converges in distribution to the law of the root neighborhood dynamics on the limit tree $G$ (see  \cites{BenjaminiSchramm,aldous-lyons} for a definition of local convergence).  When the limit graph $G$ is a regular tree  or, more generally, a unimodular Galton-Watson tree,  it is shown in \cites{Gangulythesis,GanRamTree} and \cite{GanRamUGW} that the root neighborhood marginal dynamics of the interacting process on $G$ can be autonomously characterized by a stochastic equation called the local-field equation.  
This characterization is very useful because for many classes of random graphs of interest, including the configuration model, which includes the class of random regular graphs, and Erd\"os-R\'enyi graphs, the limit graph $G$ is a unimodular Galton-Watson tree \cites{RemcoBook, GanRam2024}, and thus the local-field equation serves as an approximation of the path empirical measure of interacting jump processes on finite sparse (random) graphs. However, even when the original dynamics are Markov, in general the local-field equation describes a trajectory-dependent (non-Markovian) stochastic process that is nonlinear in the sense that its evolution depends on the law of its past history. More precisely, the autonomous characterization of the root neighborhood marginal dynamics by the local-field equation replaces the transition rates of the neighbors of the root particle by conditional rates, where the conditioning is done on the full trajectory of the root neighborhood dynamics.

A desirable goal is then to identify when the local-field equation can yield  more tractable descriptions of the limit marginal process. A natural idea is to replace the aforementioned conditioning on the full trajectory of the root neighborhood dynamics by a conditioning on the present state to obtain a Markov local-field equation \cite{GanRamTree}. For a certain class of interacting particle systems, it was shown that the solutions to the local-field equation and to the Markov local-field equation coincide.  
In \cite{cocomello2023exact}, this is established for SIR and SEIR dynamics, a sub-class of interacting particle systems widely used to model epidemics. More complex epidemics dynamics are considered in \cite{cocomello2024generalized}. A key feature of these models is a certain  pairwise linear property of the interactions: the rate at which an individual is infected is the sum of the rates at which the disease is spread by each of the individual's infected neighbors. 
The derivation of the local-field equation in \cites{GanRamTree, GanRamUGW} depended crucially on a second-order Markov random field (2-MRF) property of trajectories of interacting jump processes \cite{GanRamMRF}.
The simplification of the local-field equation for the epidemic dynamics studied in \cites{cocomello2023exact,cocomello2024generalized}
follows from a stronger conditional independence property: any two subpopulations divided by a finite single boundary of susceptible (never infected) individuals are independent of each other.

The goal of this article is to establish a tractable description of the hydrodynamic limit of a larger class of processes. 
One motivating example is a probabilistic model of neurological seizures introduced in \cite{MoosTrucc23}. In this model, active neurons can excite their neighbors causing them to activate, with inactive neurons having a hindering effect on that excitement. This results in spreading dynamics with transition rates that are not pairwise linear. An idealization of this model is given in Section \ref{subsubsec:neural_SIR}. Other applications with non pairwise linear interactions include certain majority-vote dynamics and a variant of multivariate Hawkes processes, such as the models given in Sections \ref{subsubsec:voter_model} and \ref{subsubsec:Hawkes}.

In this article, we show that for a class of dynamics exhibiting a monotonic property for their state transitions, the solution to the Markov local-field equations has the same time marginals as the solution to the local-field equations, which in turn coincide with those of the original process. Unlike previously-established results for SIR models \cite{cocomello2023exact} which required a pairwise linear structure in the interactions, the laws of the trajectories can in general differ. The Markov local-field equation thus accurately describes the empirical distribution of the particle states in a tractable way, at the cost of not capturing the dynamics of the full trajectories. The hydrodynamic limit can then be characterized through the solution to the system of Kolmogorov forward ODEs associated with the Markov process thus obtained.

Our approach relies on two crucial ingredients: a general Markovian projection result for pure jump processes, which we apply to the local-field equations, and a conditional independence property. The Markovian projection approach relies on finding a Markov process with the same marginal laws as the process of interest. This approach stems from a long line of work started by \cite{gyongy1986mimicking}, which we detail in Section \ref{subsec:proof_outline_mimicking}. In general, a Markovian projection of the local-field equations would describe the limit empirical measure process, but its transitions depend on the law of the local-field equations. We show that for a certain class of interacting particle systems, using a conditional independence property related to the graph structure for the time-marginals of the laws of the particles, it is in fact possible to obtain an autonomous tractable characterization of the hydrodynamic limit. This conditional independence property is a time-marginal version of the aforementioned second-order MRF property. In order to establish this, we first derive a trajectorial second-order MRF property up to a class of stopping times, which is of independent interest.

\subsection{Comparison with alternative approaches}
\label{sec:intro_stateoftheart}
The most common approach to obtain a tractable characterization of the hydrodynamic limit of interacting particle systems is known as the mean-field approximation. It is obtained by considering dynamics on the complete graph where interactions are scaled inversely proportionally to the number of vertices. For a large class of models, a tractable description of the hydrodynamic limit in the infinite-particle system is obtained through a McKean-Vlasov equation \cites{McKean_66, Dobrushin1979}.

A significant drawback of this approach is that the dependence on the underlying graph structure, as well as correlations between particles, are lost in the mean-field limit. A significant recent effort on the literature has been to explore different ways to circumvent this limitation of the mean-field framework. When the underlying graph is dense, the properties of graphons have been used to derive new limit equations \cites{ZAN_2022, jabin2022}, that take the form of an infinite system of ODEs. An alternative approach to describe large interacting particle systems is based on the so-called Poisson Hypothesis. Popularized by Kleinrock for large queueing systems \cite{Kleinrock_1}, this prescribes that the flow of arrivals to a given node can be approximated by a Poisson flow. To extend this to agent-based models, one interprets the flow of arrivals in a queue as the effect of interactions on a given particle. Under the Poisson Hypothesis, the behavior of each particle is described by a stochastic differential equation (SDE), but the particles are considered independent and interaction times are replaced by Poisson processes. This allows for tractability in certain models such as certain queueing models \cite{Vladimirov_2018} and intensity-based models from computational neuroscience \cites{Baccelli_2019,Dav24}. 

\subsection{Organization of the article}

The rest of this article is organized as follows: in Section \ref{sec:main_result}, we introduce the general class of dynamics we will be considering and state our main result. In Section \ref{sec:ram_and_simu}, we describe how to use our result to approximate dynamics on a given graph and provide examples of applications, illustrated with simulations. In Section \ref{sec:proof_outline}, we restate our result and assumptions in a technically tighter fashion and give an outline of the proof, which relies on establishing a time-marginal conditional independence property and the well-posedness of a Markovian projection of the dynamics. In Sections \ref{sec:proof_2-MRF(t)} and \ref{sec:Markov_proj_proof}, we detail the proofs of these two results. In Section \ref{sec:ode_well-posed} we establish the well-posedness of the ODE used in our main result.

\subsection{Notation}
\label{sec:notation} We briefly overview common notation used throughout the paper. We use $G=(V,E)$ to denote a graph with vertex set $V$ and edge set $E$. When clear from context, we identify a graph with its vertex set, and so for a vertex $v$ we might write $v\in G$ instead of the more accurate $v\in V$.
 Given $v\in G$, we write $\partial_v:=\{w\in V : \{v,w\}\in E\}$ for the \textit{neighbors} of $v$. We also define $\cl_v:=\partial_v \cup \{v\}$. The \textit{degree} of a vertex is defined as $d_v:=|\partial_v  |$.

Given a set $\Ymc$, a configuration $y\in \Ymc^V$ 
 and $A\subset V$, we write $y_A:= \{y_v\ : \ v\in A\}$, and we follow the same convention for random elements $Y$ in $\Ymc^V$. Given a probability space $(\Omega,\Fmc,\P)$, we denote by $\law(Y)$ the law of a $\Omega-$valued random variable $Y$. For random variables $Y,$ $U,$ $W$, we write $U\perp Y$ to mean that $U$ and $Y$ are independent, and $ U\perp Y | W$ to mean that $U$ and $Y$ are conditionally independent given $W$.
 We denote $x\wedge y$ (resp. $x\vee y$) the minimum (resp. the maximum) of $x$ and $y.$

Given an interval $\Smc \subset \R$  and $\Mmc$ a metric space, let $\Dmc(\Smc:\Mmc)$ be the space of càdlàg functions equipped with the Skorokhod topology. Throughout, we will fix a set $\stateS\subset \N$ and we write $\Dmc:=\Dmc([0,\infty):\stateS)$ and for a set $A$, $\Dmc^A:=\Dmc([0,\infty):\stateS^A)$.

\section{Main result}
\label{sec:main_result}
Our main result, Theorem \ref{thm:main}, provides a tractable ODE characterization of the time marginals of the typical particle neighborhood of a class of Markov pure-jump interacting particle systems (IPS) on Galton-Watson trees. As a corollary, we obtain a characterization of the hydrodynamic limit of the path empirical measure of a sequence of IPS on finite-size configuration model graphs.

\subsection{Model description and basic assumptions}
We consider interacting particle systems (IPS) on unimodular Galton-Watson (UGW) trees, which are of interest since they arise as local limits of many random graph sequences that model real-world networks, including sparse \Erdos\ graphs and
configuration models. We first define the class of UGW trees we consider.

\begin{assumption}[Graph structure]
\label{Ass:A_1}
$G$ is a unimodular Galton-Watson tree with offspring distribution $\theta$ satisfying $|\Supp(\theta)|<\infty$, that is, its root has offspring distribution $\theta$ and each vertex of all subsequent generations has a number of offspring which is independent of the degree of other vertices in the same or previous generations, sampled according to the \textit{size-biased} distribution $\hat{\theta}$  given by
\begin{equation}
    \label{eq:UGW_offspring_dist}
    \hat{\theta}(k)=\frac{(k+1)\theta(k+1)}{\sum_{n\geq 1}n\theta(n)}, k \in \N.
\end{equation}
\end{assumption}

Unimodularity is a symmetry property that can be roughly understood as requiring that the random, potentially infinite graph $G$ looks the same from any vertex. A rigorous definition of unimodularity is deferred to Section \ref{subsec:unimod} to lighten exposition.

To describe IPS on random graphs, it will be useful to introduce a labeling of the vertex set, known as the   Ulam-Harris-Neveu labeling, which identifies a realization $\Tmc$ of the UGW with subgraph of the graph of all possible vertices. The latter has vertex set $\V:=\{\root\} \cup (\cup_{k=1}^\infty \N^k)$, where $\root$ denotes the root, and edges $\{\{v,vi\}: v\in \V, i\in\N\}$, where $vi$ denotes concatenation, with the convention $\root u =u \root =u$ for all $u\in\V$. 

A tree $T$ with root $\root_{T}$ is identified (uniquely up to root-preserving automorphisms) to a subgraph of $\V$ via a map $\Vmc$ from the vertex set of $T$ to $\V$ such that 
\begin{enumerate}[label=(\roman*)]
    \item $\Vmc(\root_T)=\root$;
    \item $\Vmc(\partial_{\root_T})= \{m\in\N \ : \ m\leq d_{\root_T}\}$;
    \item for $v
    \in\Tmc$ at graph distance $k\in\N$ from $\root_T$, $\Vmc(v)=u \in\N^{k}$ and $\Vmc(\partial_v)=\{\pi_v\}\cup\{vm \ : \ m\in\N,  m\leq d_v-1\}$, where $\pi_v$ is the unique $w\in\V$ such that there exists $k\in \N$ with $wk=v$.
    \end{enumerate}

We focus on Markov IPS consisting of a collection of pure jump càdlàg stochastic processes indexed by the set $\V$, where each process describes the evolution of a particle that takes values in a finite set $\stateS$, and whose allowable transitions as jumps lie in a set $\Jmc\subset \stateS - \stateS\setminus\{0\}$. Without loss of generality, we assume that $\stateS\subset\natzero$. In order to represent marks on a tree $T \subset \V$ we consider a new mark $\extra$, and define $\stateS_{\extra}:=\stateS\cup\{\extra\}$. Given $x\in\stateS^T$, we extend it to an element in $ \stateS_{\extra}^T$  by setting $x_w=\extra$ for all $w\in\V \setminus T$. 

For reasons that will become apparent later, we will find it convenient to characterize the IPS as the solution to a jump SDE, rather than via its generator.

\begin{definition}\label{def:IPS}
    Let $\tree$ be a UGW tree with random vertex set $V\subset \V$ (under the Ulam-Harris-Neveu labeling), and random edge set $E$. We define the process $X=\{X_v\}_{v\in \V}$ as the solution of the following SDE:

\begin{equation}
\label{eq:IPS_SDE}
    X_v(t)= X_v(0) + \sum_{j \in \Jmc}\int_{(0,t)\times \R_+} j  \one_{\{u<\r^{j}(s,X_{v}(s-),X_{\partial v}(s-))\}} \mathbf{N}^j_v(\diff s,\diff u),\  v\in \V,
\end{equation}
where 
\begin{itemize}

    \item for each  $j \in \Jmc,$ the rate $\r^j:[0,\infty)\times \stateS_\extra  \times \cup_{n=0}^{\dmax} \stateS^n \rightarrow [0,\infty)$ is a measurable mapping that represents the size $j$ jump intensity of particle $v.$ 

    \item $(\mathbf{N}^j_v)_{v\in \V, j\in\Jmc}$ are independent Poisson point processes on $\R_+\times\R_+$ with unit intensity that are also independent of the graph $\tree$ and the initial conditions $(X_v(0))_{v\in \V}$.
\end{itemize}
We assume that $\r^j(t,\extra, y)=0$ for all  $t\in[0,\infty)$ and $y\in\cup_n\stateS^\N$, and that $\r^j(t,a,y)$ is invariant under permutations of $y$.
We set $X_v(0)=\extra$ whenever $v\in\V\setminus\{\tree\}.$
\end{definition}
Note that the SDE \eqref{eq:IPS_SDE} describes a class of processes that are locally interacting with respect to $\tree$ 
in the sense that at each time $s\geq0$, the jump rate of vertex $v$  depends on the state $X(s)$ only through the state $X_v(s)$ of $v$ at time $s$, and the (unordered) states $X_{\partial_v}(s):=\{X_w(s): vw\in E\}$ of its neighbors in $\tree$.

We now state some general sufficient conditions on the initial conditions and transition rates under which the SDE \eqref{eq:IPS_SDE} is well-posed.

Given a graph $G=(V,E)$, we define the \textit{double boundary} of a set $A\subset V$ as  \begin{equation}\label{eq:double_boundary}
\partial^2 A := \left\{v\in V\setminus A : \exists u \in  \text{ with }d_G(v,u) \leq 2 \right\},
\end{equation}
where $d_G$ denotes the graph distance between two vertices.
\begin{definition}[Semi-global Markov Random Field]
\label{def:2-MRF}
    Fix a graph $G=(V,E)$ and a state space $\Ymc$. A  $\Ymc^V$-valued random element $Y$ is said to be a \textit{semi-global Markov Random Field of order 2} (henceforth referred to as the 2-MRF) if, for every $A\subset V$ such that $|\partial^2A|<\infty$, we have
    \begin{equation}
    \label{eq:2-MRF}
        Y_A \perp Y_{V\setminus \{A\cup \partial^2 A\}} | Y_{\partial^2 A}.
    \end{equation}
   
\end{definition}
\begin{assumption}[Initial conditions]
    \label{Ass:B_1} The following two conditions are satisfied.
    \begin{enumerate}
        \item The initial conditions $X(0)$ form a 2-MRF in the sense of Definition \ref{def:2-MRF}.
        \item $(\tree, X_{\tree}(0))$ is unimodular.
    \end{enumerate}
.
\end{assumption}
We refer to Section \ref{subsec:unimod} for a precise definition of unimodularity.


The next assumption concerns the transition rates $\r^j$ appearing in \eqref{eq:IPS_SDE}.
\begin{assumption}[Rates]
\label{Ass:A_2}
For every $j\in\Jmc$ the rates  $\{\r^{j}\}_{j\in\Jmc}:[0,\infty)\times \stateS\times\cup_{n=0}^\infty \stateS^n \rightarrow [0,\infty)$ are c\`agl\`ad (left-continuous with finite right limits) and satisfy 
\begin{equation}
\label{eq:rates_bound}
\r^j(t, \cdot ) \leq C(d_v+1,t),
\end{equation}
where 
\begin{enumerate}
    \item $d_v=\left|\set{u\in V : uv\in E}\right|$ is the degree of vertex $v$ in $\tree$.
    \item $C : \N \times \R_+ \rightarrow \R_+$ is a function that is non decreasing in each of its
arguments and satisfies $\lim_{t\rightarrow\infty} C(d,t)<\infty$ for all $d\in\N$.
\end{enumerate}
\end{assumption}

The c\`agl\`ad assumption on the rates in Assumption \ref{Ass:A_2} ensures that the trajectories of $X$ in \eqref{eq:IPS_SDE} are c\`adl\`ag.

The well-posedness of \eqref{eq:IPS_SDE} follows from the following result from \cite{GanRam2024}:

\begin{proposition}[\cite{GanRam2024}*{Theorem 4.3}]
\label{prop:SDE-wellposed}
Under Assumptions \ref{Ass:A_1}, \ref{Ass:B_1} and \ref{Ass:A_2}, the SDE \eqref{eq:IPS_SDE} is strongly well-posed, in the sense that there exists at least one weak solution, and the SDE is pathwise unique.
\end{proposition}

\subsection{A class of IPS}
Our main result applies to a class of IPS whose transition rates exhibit a certain monotonic structure. To describe this structure, we introduce the notion of a transition graph.

\begin{definition}[Transition Graph]\label{def:trans_graph}
    Given jump rates $\{\r^j\}_{j\in\Jmc}$ as in \eqref{eq:IPS_SDE}, the corresponding (directed) transition graph $G_{\r}=(\stateS, \Emc_{\rho})$ has vertex set $\stateS$ and directed-edge set
    \begin{equation*}
        \Emc_{\r} := \bigcup_{a\in\stateS, j\in\Jmc}\left\{(a,a+j) :  \exists t\in\R_+, x\in \cup_{n=0}^{\dmax}\stateS_\extra^{n} \text{ with } \r^j(t,a,x)>0.\right \}
    \end{equation*}
\end{definition}
We now state our main assumption on the IPS:
\begin{assumption}[Monotonic particle transitions]
\label{Ass:A_3}
    The transition rates $\r=\{\r^j\}_{j\in\Jmc}$ are such that the corresponding transition graph $G_{\r}$ specified in Definition \ref{def:trans_graph} is a finite directed acyclic graph.
\end{assumption}

\begin{remark}
\label{rem:alt_def:ass_a3}
An alternative formulation of Assumption \ref{Ass:A_3} is that for all $a,\ b\in\stateS$, if there exists $s,t\in\R_+$ such that 
\begin{equation}
\label{eq:monotonicity_equiv}
\P(X^G_v(t+s)=b |X^G_v(t)=a)>0,
\end{equation}
then \begin{equation*}
    \P(X^G_v(t'+s')=a |X^G_v(t')=b)=0,\ \ \forall t',s'\in\R_+.
\end{equation*}
This can be expressed in terms of the rates $\{\rho^j\}_{j\in\Jmc}$ by requiring that whenever $a,\ b$ are such that there exist $j_i\in\Jmc$, $a_i\in\stateS$, $t_i\in\R_+$, $x_i\in \stateS_{\extra}^{\dmax}$, $i=1,...,m$   so that $a_1=a$, $a_{i+1}=a_i+j_i$, $a_m=b$ and $\rho^{j_i}(t_i,a_i,x_i)>0$, then $\rho^{b-a}(s,a,y)=0$ for all $s\in\R_+$ and $y\in\stateS_{\extra}^{\dmax}$. 
\end{remark}

Assumption \ref{Ass:A_3} is satisfied by many interesting dynamics arising in various applications. This includes compartmental epidemiological models such as SIR models (wherein vertices represent individuals in a population, and can be either susceptible to, infected by, or recovered from an infectious disease) and some variations thereof whose transition graphs are depicted in Figure 1a. Another class of examples are entrenched majority voter models, where undecided vertices can change opinions according to the majority opinion of their neighborhood (see Figure 1b. for the associated transition graph). Assumption \ref{Ass:A_3} is also satisfied by more complicated models with nonlinear transitions, such as a model of seizure-propagation dynamics given in Section \ref{subsubsec:neural_SIR}. Assumption \ref{Ass:A_3} also requires the state space to be finite, which precludes some $\N$-valued dynamics such as multivariate Markov Hawkes processes. A workaround is to impose a deterministic threshold, see Section \ref{subsubsec:Hawkes}. 

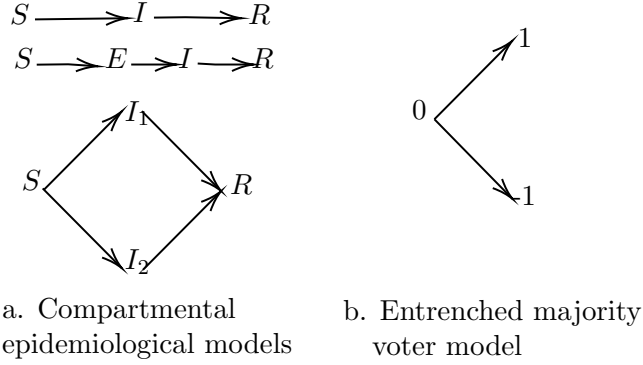
\begin{figure}
\centering
\tikzset{every picture/.style={line width=0.75pt}} 

\begin{tikzpicture}[x=0.75pt,y=0.75pt,yscale=-1,xscale=1]

\draw    (391.57,77.09) -- (429.6,115.92) ;
\draw [shift={(431,117.35)}, rotate = 225.6] [color={rgb, 255:red, 0; green, 0; blue, 0 }  ][line width=0.75]    (10.93,-3.29) .. controls (6.95,-1.4) and (3.31,-0.3) .. (0,0) .. controls (3.31,0.3) and (6.95,1.4) .. (10.93,3.29)   ;
\draw    (391.57,77.09) -- (429.59,38.77) ;
\draw [shift={(431,37.35)}, rotate = 134.77] [color={rgb, 255:red, 0; green, 0; blue, 0 }  ][line width=0.75]    (10.93,-3.29) .. controls (6.95,-1.4) and (3.31,-0.3) .. (0,0) .. controls (3.31,0.3) and (6.95,1.4) .. (10.93,3.29)   ;
\draw    (190.12,26.69) -- (206,26.49) -- (234.88,26.13) ;
\draw [shift={(236.88,26.1)}, rotate = 179.28] [color={rgb, 255:red, 0; green, 0; blue, 0 }  ][line width=0.75]    (10.93,-3.29) .. controls (6.95,-1.4) and (3.31,-0.3) .. (0,0) .. controls (3.31,0.3) and (6.95,1.4) .. (10.93,3.29)   ;
\draw    (250.12,25.79) -- (266,25.59) -- (292.24,26.07) ;
\draw [shift={(294.24,26.1)}, rotate = 181.04] [color={rgb, 255:red, 0; green, 0; blue, 0 }  ][line width=0.75]    (10.93,-3.29) .. controls (6.95,-1.4) and (3.31,-0.3) .. (0,0) .. controls (3.31,0.3) and (6.95,1.4) .. (10.93,3.29)   ;
\draw    (194,112) -- (232.59,150.93) ;
\draw [shift={(234,152.35)}, rotate = 225.25] [color={rgb, 255:red, 0; green, 0; blue, 0 }  ][line width=0.75]    (10.93,-3.29) .. controls (6.95,-1.4) and (3.31,-0.3) .. (0,0) .. controls (3.31,0.3) and (6.95,1.4) .. (10.93,3.29)   ;
\draw    (244,73) -- (282.59,111.93) ;
\draw [shift={(284,113.35)}, rotate = 225.25] [color={rgb, 255:red, 0; green, 0; blue, 0 }  ][line width=0.75]    (10.93,-3.29) .. controls (6.95,-1.4) and (3.31,-0.3) .. (0,0) .. controls (3.31,0.3) and (6.95,1.4) .. (10.93,3.29)   ;
\draw    (194.57,112.74) -- (232.59,74.42) ;
\draw [shift={(234,73)}, rotate = 134.77] [color={rgb, 255:red, 0; green, 0; blue, 0 }  ][line width=0.75]    (10.93,-3.29) .. controls (6.95,-1.4) and (3.31,-0.3) .. (0,0) .. controls (3.31,0.3) and (6.95,1.4) .. (10.93,3.29)   ;
\draw    (244.57,153.09) -- (282.59,114.77) ;
\draw [shift={(284,113.35)}, rotate = 134.77] [color={rgb, 255:red, 0; green, 0; blue, 0 }  ][line width=0.75]    (10.93,-3.29) .. controls (6.95,-1.4) and (3.31,-0.3) .. (0,0) .. controls (3.31,0.3) and (6.95,1.4) .. (10.93,3.29)   ;
\draw    (238.65,49.47) -- (259.59,49.47) ;
\draw [shift={(261.59,49.47)}, rotate = 180] [color={rgb, 255:red, 0; green, 0; blue, 0 }  ][line width=0.75]    (10.93,-3.29) .. controls (6.95,-1.4) and (3.31,-0.3) .. (0,0) .. controls (3.31,0.3) and (6.95,1.4) .. (10.93,3.29)   ;
\draw    (272.18,48.96) -- (281,49.47) -- (298.41,49.47) ;
\draw [shift={(300.41,49.47)}, rotate = 180] [color={rgb, 255:red, 0; green, 0; blue, 0 }  ][line width=0.75]    (10.93,-3.29) .. controls (6.95,-1.4) and (3.31,-0.3) .. (0,0) .. controls (3.31,0.3) and (6.95,1.4) .. (10.93,3.29)   ;
\draw    (191,49.54) -- (206.89,49.35) -- (219.89,50.24) ;
\draw [shift={(221.88,50.37)}, rotate = 183.92] [color={rgb, 255:red, 0; green, 0; blue, 0 }  ][line width=0.75]    (10.93,-3.29) .. controls (6.95,-1.4) and (3.31,-0.3) .. (0,0) .. controls (3.31,0.3) and (6.95,1.4) .. (10.93,3.29)   ;

\draw (378.78,65.98) node [anchor=north west][inner sep=0.75pt]   [align=left] {0};
\draw (432.21,29.75) node [anchor=north west][inner sep=0.75pt]   [align=left] {1};
\draw (429.21,109.6) node [anchor=north west][inner sep=0.75pt]   [align=left] {\mbox{-}1};
\draw (344,167) node [anchor=north west][inner sep=0.75pt]   [align=left] {b. Entrenched majority \\ \ \ \ voter model};
\draw (234,66.4) node [anchor=north west][inner sep=0.75pt]    {$I_{1}$};
\draw (234,141.4) node [anchor=north west][inner sep=0.75pt]    {$I_{2}$};
\draw (287,103.4) node [anchor=north west][inner sep=0.75pt]    {$R$};
\draw (296.12,17.29) node [anchor=north west][inner sep=0.75pt]    {$R$};
\draw (237.94,16.39) node [anchor=north west][inner sep=0.75pt]    {$I$};
\draw (182,101.4) node [anchor=north west][inner sep=0.75pt]    {$S$};
\draw (176.18,17.29) node [anchor=north west][inner sep=0.75pt]    {$S$};
\draw (172,166) node [anchor=north west][inner sep=0.75pt]   [align=left] {a. Compartmental \\epidemiological models};
\draw (297.88,39.76) node [anchor=north west][inner sep=0.75pt]    {$R$};
\draw (260.88,39.76) node [anchor=north west][inner sep=0.75pt]    {$I$};
\draw (177.94,39.76) node [anchor=north west][inner sep=0.75pt]    {$S$};
\draw (223.82,39.76) node [anchor=north west][inner sep=0.75pt]    {$E$};

\end{tikzpicture}

\caption{State space representation for some example dynamics}
\end{figure}

\subsection{Main result: tractable description of space-time marginal dynamics}
\label{subsec:main_result_statement}
For the class of IPS satisfying Assumptions \ref{Ass:A_1}, \ref{Ass:B_1}, \ref{Ass:A_2} and \ref{Ass:A_3}, our main contribution is to provide an ODE description of the evolution of the law of the root particle and its neighbors. We start by establishing well-posedness of this system of ODEs.

We define $\V_1:=\{\root\}\cup \N \subset \V$. Given $\offspring\in\Pmc(\natzero)$, define $\supoff:=\{k: \offspring(k)>0\}$, and suppose that $|\supoff|<\infty$. We define 
\begin{equation}\label{eq:d_max}
    \dmax:=\dmax(\offspring):=\max\{k: \offspring(k)>0\}
\end{equation}
and
\begin{equation}
    \V_1^\offspring := \{\root\}\cup \{k \in \N, 1\leq k<\dmax(\offspring)\}.
\end{equation}
Note that the configuration space of the neighborhood of the root is then 
\begin{equation*}
C^{\offspring}:=\cup_{n\in\supoff}(\stateS^{n+1}\times\{\extra\}^{\dmax-n})\subset \stateS_\extra^{\V_1^\offspring}.
\end{equation*}
For $\vec{a}\in \stateS_\extra^{\V_1^\offspring}$, we define 
\begin{equation}\label{eq:degree_of_vec}
    k(\vec{a}):= \one_{\{\vec{a}\in C^\offspring\}}(\max\{v : a_v\neq \extra\}-1),
\end{equation}
that is, $k(\vec{a})$ is the degree of $\root$ when $X_{\V_1^\offspring}=\vec{a}$.
We let 
\begin{equation*}
    \Pmc^{\offspring}:= \set{ p \in \Pmc(\stateS_\extra^{\V_1^{\offspring}}) \ :\ p(C^\offspring)=1}.
\end{equation*}


In the following, having identified $\stateS$ with a subset of $\N$, we let $(\ev)_{v=0,1,...,\dmax}$, be standard basis vectors.
To simplify exposition in the following results, we define $\rr^j:[0,\infty)\times (\Z\cup\{\extra\})^{\dmax+1} \rightarrow [0,\infty)$  by 
\begin{equation}\label{def:deparametrize_r}
    \rr^j(t,y):=\begin{cases}
        \r^j\left(t,y_0,\{y_{v}\}_{v=1}^{\dmax}\right) & \text{if } y\in C^{\offspring}  
        \\ 0 & \text{otherwise,}
    \end{cases}
\end{equation}
that is, $\rr^j(t,y)$ is a reparametrization of the jump rates $\r^j$ in Definition \ref{def:IPS}  when $y$ is in the allowed configuration $C^\offspring$, and $\rr(t,y)=0$ otherwise.
We also adopt the conventions that $\extra+m=\extra$ for all $m\in\Z$, and that $0/0=0$.
As is standard practice, we use the dot notation for derivatives with respect to time.

\begin{proposition}\label{prop:ODE-wellposed}
    Suppose that Assumptions \ref{Ass:A_1},\ref{Ass:B_1},\ref{Ass:A_2} and \ref{Ass:A_3} hold, with $\offspring$ the degree distribution from Assumption \ref{Ass:A_1}. Fix $p\in \Pmc^\offspring$.  Then there exists a unique solution $\lawMF=\lawMF^\offspring$ to the following system of ODEs:

    
\begin{align}
     \label{eq:MF-ODE-UGW}
   & \dot \lawMF_{t}(\vec{a}) = \sum_{j\in\Jmc}\sum_{v=0}^{k(\vec{a})}\left(  \lawMF_{t}(\vec{a}-j\ev) \Psi_v^{j,j}(t,\lawMF_t, \vec{a})- \lawMF_{t}(\vec{a}) \Psi_v^{0,j}(t,\lawMF_t, \vec{a}) \right),
\\ \label{eq:MF-ODE-UGW_initial}
&\lawMF_{0}(\vec{a})= p(\vec{a}),
\end{align}
for all  $\vec{a}\in \stateS_\extra^{\V_1^\offspring}$, where $\Psi_v^{\ell,j}: [0,\infty)\times \Pmc^{\offspring}\times \stateS_\extra^{\V_1^{\offspring}}\rightarrow [0,\infty)$ is given by
\begin{equation*}
    \Psi_v^{\ell,j}(t,f,\vec{a}):= 
    \begin{cases}
\rr^j(t,\vec{a}-\ell e_{\root}) & \text{if } v=\root,  \\ 
\frac{\sum_{\vec{b}\in C^\offspring} k(\vec{b})\rr^j(t,\vec{b}) f(\vec{b}) \one_{\{ b_\root=a_v-\ell,b_1=a_\root\}}}{\sum_{\vec{c}\in C^\offspring} k(\vec{c}) f(\vec{c}) \one_{\{ c_\root=a_v-\ell,c_1=a_\root\}}} & \text{otherwise,}
\end{cases}
\end{equation*}
    with $\rr^j$ as defined in \eqref{def:deparametrize_r}.
\end{proposition}
 Proposition \ref{prop:ODE-wellposed} is proved in Section \ref{sec:ode_well-posed} by showing that the right-hand side of \eqref{eq:MF-ODE-UGW} is Lipshitz in $\lawMF_t\in \Pmc^\offspring$. When $\offspring=\delta_{\kappa}$ for $\kappa\in\N$, $\kappa\geq 2$,  the graph $\tree$ is the $\kappa$-regular tree, and we obtain the following corollary.
\begin{corollary}
   Suppose that Assumptions \ref{Ass:A_1},\ref{Ass:B_1},\ref{Ass:A_2} and \ref{Ass:A_3} hold with $\offspring=\delta_\kappa$ for $\kappa\in \N, \kappa\geq 2$. Fix $p\in \Pmc^{\delta_k}$. Then there exists a unique solution $\lawMF=\lawMF^\offspring$ to the following system of ODEs:
    \begin{align}
 \label{eq:MF-ODE}
  &  \dot \lawMF_t(\vec{a}) = \sum_{j\in\Jmc}\sum_{v=0}^{\kappa}\left(  \lawMF_t(\vec{a})-je_v) \Psi_v^{j,j}(t,\lawMF_t, \vec{a})- \lawMF_t(\vec{a}) \Psi_v^{0,j}(t,\lawMF_t, \vec{a}) \right),
        \\ \label{eq:MF-ODE-initial}
&\lawMF_{0}(\vec{a})= p(\vec{a})
    \end{align}
    for all $\vec{a}\in\stateS^{\kappa+1}$
where  $\Psi_v^{\ell,j}: [0,\infty)\times \Pmc^{\delta_{\kappa}}\times \stateS^{\kappa+1}\rightarrow [0,\infty)$ is given by
\begin{equation*}
    \Psi_v^{\ell,j}(t,f,\vec{a}):= 
    \begin{cases}
\rr^j(t,\vec{a}-\ell \vec{e^{v}}) & \text{if } v=\root,  \\ 
\frac{\sum_{\vec{b}\in\stateS^{\kappa+1}} \rr^j(t,\vec{b}) f(\vec{b}) \one_{\{ b_\root=a_v-\ell,b_1=a_\root\}}}{\sum_{\vec{c}\in\stateS^{\kappa+1}} f(\vec{c}) \one_{\{ c_\root=a_v-\ell,c_1=a_\root\}}} & \text{otherwise,}
    \end{cases}
\end{equation*}
with $\rr^j$ as defined in \eqref{def:deparametrize_r}.
\end{corollary}

We now state our main result.
\begin{theorem} \label{thm:main}
Suppose that Assumptions \ref{Ass:A_1},\ref{Ass:B_1},\ref{Ass:A_2} and \ref{Ass:A_3} hold.
Let $\lawMF$ be the solution of the ODE  \eqref{eq:MF-ODE-UGW}-\eqref{eq:MF-ODE-UGW_initial} with initial conditions $p=\law(X_{\V_1^\offspring}(0))$.Then, for every $t\geq 0,$
\begin{equation*}
    \lawMF_t=\law(X_{\V_1^\offspring}(t)).
\end{equation*}
\end{theorem}
Theorem \ref{thm:main} is proved in Section \ref{sec:proof_outline}. We now provide an outline of our proof techniques. Our starting point is an autonomous characterization of the law $\law(X^\Tmc_{\V_1^\offspring})$ of the particle at the root and its neighbors as the unique solution to a path-dependent jump SDE called the \textit{local-field equation}, which was derived in \cite{GanRamUGW} using a conditional independence property of the trajectories of $X$ and the graph structure established in \cite{GanRamMRF}. In general the law of $X^\Tmc_{\V_1^\offspring}$ need not be Markov, and its law can be hard to compute. 
However, we show that under Assumption \ref{Ass:A_3}, the time marginals of $X$ satisfies a stronger conditional independence property (see Section \ref{subsec:proof_outline_2-MRF(t)}). We combine this property with a Markovian projection result (see Section \ref{subsec:proof_outline_mimicking}) to characterize the space-time marginals $X_{\V_1^\offspring}(t)$ as the solution of a Markovian SDE called the \textit{Markov local-field equation}. The ODE \eqref{eq:MF-ODE-UGW}-\eqref{eq:MF-ODE-UGW_initial} can then characterize $X_{\V_1^\offspring}(t)$ by solving an associated Kolmogorov forward equation.

\begin{remark}
By Assumption \ref{Ass:B_1}, the law  $\lawMF_0(\vec{a})=\law(X_{\V_1^\offspring})$ is invariant under the reordering of $a_1,...,a_l$ with $l=K(\vec{a})$.
Since the rates $\r^j(t,x,\vec{b})=\rr^j(t (x,\vec{b}))$ are invariant under reordering of $\vec{b}$ (See Definition \ref{def:IPS}), then $\lawMF_t(\vec{a})$ is also invariant under the reordering of $a_1,...,a_{\kappa}$ with $\kappa=k(\vec{a})$.
In fact, in the special case of the $\kappa$-regular tree,  \cite{Gangulythesis}*{Proposition 6.10} establishes that, whenever the initial conditions $X(0)$ are invariant under tree-automorphisms, so is the process $X$. Therefore, for every permutation $\sigma$ of $\{0,...,m\}$ satisfying $\sigma(0)=0$, we have
\begin{equation*}
\lawMF_t(a_{\sigma(v)}, v=0,...,l)=\lawMF_t(\vec{a}).    
\end{equation*}

Letting $m=|\stateS|$, in the case when $\offspring=\delta_\kappa$, we can rewrite equation \eqref{eq:MF-ODE} as a system of $m\binom{\kappa+m-1}{m-1}$ equations corresponding to $m$ possible values of $a_\root$, and $\binom{\kappa+m-1}{m-1}$ unordered combinations of $a_1,...,a_\kappa$. This is a slight reduction of dimensionality compared to the original system of $m^{\kappa+1}$ equations. For general, non-degenerate $\offspring$ with support $\supoff$, by a similar argument we obtain a system of $m \sum_{k\in \supoff }\binom{k+m-1}{m-1}$ equations. 
\end{remark}

\subsection{A limit theorem formulation of the main result}
The UGW($\offspring$) tree arises at the local limit of the so-called configuration model graphs (which we define below). Therefore, invoking  a result of Ganguly and Ramanan \cite{GanRam2024} restated below as Proposition \ref{prop:hydrodynamic}, our Theorem \ref{thm:main} describes the limit, as the size of the graph goes to infinity, of the ensemble behavior of IPS dynamics on configuration models, in a sense that we now make precise.
\begin{definition}[Configuration model]\label{def:CM}
    Fix $n\in\natzero$ and let $\vec{d}=\{d_{i}\}_{i=1}^n$ be a graphical sequence\footnote{A vector in $\natzero^n$ such that there exists $G=(V,E)$ with vertices $V=\{1,2,...n\}$ where each $v\in V$ has degree $d_v$.},
    The configuration model $\CM_n(\vec{d})$ is a random graph sampled uniformly at random among the graphs on $n$ vertices with degree sequence $\{d_{i}\}_{i=1}^n$. 
\end{definition}

We can define IPS dynamics on a finite random graph with a deterministic number of vertices equivalently to \eqref{eq:IPS_SDE}. This enables us to reframe Theorem \ref{thm:main} as a limit theorem for a sequence of IPS on finite-size configuration model graphs. 
Given a finite graph $G$, we define
\begin{align}\label{def:fractions}
    \begin{split}
        & \emp^G_t(a):=\frac{1}{|G|}\sum_{v\in G} \one_{\set{X^G_v(t)=a}}, \quad t\in [0,\infty), a\in\stateS,
    \end{split}
\end{align}
that is, $\emp^G_t$ is the empirical distribution of $X^G(t)$.

\begin{corollary}
\label{cor:limit_thm_CM}
For $n\in \N$ let $\vec{D}^{(n)}\in \natzero^n $ be a graphical sequence such that $\sum_{i=1}^n \delta_{D_{i}^{(n)}}$ converges weakly to $\offspring$ as $n\rightarrow\infty$, and suppose that $\offspring$ has finite support.  Let $G_n=\text{CM}_n(\vec{D}^{(n)})$ as given in Definition \ref{def:CM}. Then for every $f:\stateS_\extra^{\V_1^\offspring}\rightarrow \R$,
\begin{equation}
\label{eq:empirical_distrib_cv}
\frac{1}{n}\sum_{v\in G_n} f(X_{\cl_v}^{G_n}(t)) \xrightarrow{p}\sum_{\vec{a} \in \stateS_\extra^{\V_1^\offspring} } f(\vec{a})\lawMF^\offspring_t(\vec{a}), \qquad  \text{as }n\rightarrow\infty,
\end{equation}
where $\cl_v=\{v\}\cup\partial_v$ and $\lawMF^\offspring$ is the solution to \eqref{eq:MF-ODE-UGW}-\eqref{eq:MF-ODE-UGW_initial} with initial conditions $p=\law(X_{\V_1^\offspring}(0))$. In particular, defining $\emp^{\offspring}_t\in \Pmc(\stateS)$ by $\emp^\offspring_t(a)=\sum_{\vec{b}\in \stateS^{\V_1^\offspring}} \one_{\{\vec{b}_0=a\}}\lawMF_t({\vec{b}})$, we have that 
\begin{equation}
\label{eq:main_limit_thm}
\emp^{G_n}_t \xrightarrow{p}\emp^{\infty}_t. 
\end{equation}
as $n\rightarrow\infty$.
\end{corollary}

Corollary \ref{cor:limit_thm_CM} follows from Theorem \ref{thm:main} by invoking a hydrodynamic limit result established in \cite{Gangulythesis}*{Corollary 4.7} which shows that $\emp_t^{G_n}$ converges to $\law(X_\root^{G_\infty}(t))$ whenever the sequence of finite graphs $G_n$ converges in a certain sense (\textit{local convergence in probability}, as defined in \cite{vanderHofstad2024vol2}*{Definition 2.11}; see also \cite{lacker2023local}*{Definition 2.2}  and \cite{GanRam2024}*{Definition 4.8}). The local weak convergence of the configuration model to the UGW tree is well known, see for instance \cite{vanderHofstad2024vol2}*{Theorem 41}. In fact, \cite{Gangulythesis}*{Corollary 4.7} establishes the stronger result that empirical distribution of the trajectories $X^{G_n}[t)$ converges to the law of $X^{G_\infty}_{\root}[t)$ for all $t\in\R_+$. We restate here the hydrodynamic limit for the empirical distribution of the states at time $t$, in the special case of configuration model graph and i.i.d. initial conditions. 
\begin{proposition}[Convergence of  empirical measures, \cite{GanRam2024}*{Corollary 4.12}]
     \label{prop:hydrodynamic} 
     For $n\in \N$ let $\vec{d}^{(n)}\in \natzero^n $ be a graphical sequence such that $\sum_{i=1}^n \delta_{d_{i}^{(n)}}$ converges weakly to $\offspring$ as $n\rightarrow\infty$, and $\offspring$ has finite support.  Let $G_n=\text{CM}_n(\vec{d}^{(n)})$ as given in Definition \ref{def:CM}.  
     Let $X^{G_n}$ be a IPS on $G_n$ with rates $\r=\{\r^j\}$ satisfying Assumption \ref{Ass:A_2}, and i.i.d initial conditions with marginals $p\in \Pmc(\stateS)$. Let $X$ be the IPS on the UGW($\offspring$) tree given by \eqref{eq:IPS_SDE} with i.i.d initial conditions with law $p$.  Then, for all $t\in\R_+$ and $f: \stateS_\extra^{d_{max}}\rightarrow \R_+$,
     \begin{equation}
     \label{eq:limit_emp_measure}
     \frac{1}{|G_n|}\sum_{v\in G_n}f(X_{\cl_v}^{G_n}(t))\xrightarrow{p} f(X_{\partial_\root}(t)),  \quad \text{as } n\rightarrow\infty.   
     \end{equation}

     In particular, 
\begin{equation*}
   \emp_t^{G_n}\xrightarrow{p} \law(X_\root), \qquad \text{as } n\rightarrow\infty,
\end{equation*}
where the empirical distribution $\emp_t^G$ is defined in \eqref{def:fractions}.
\end{proposition}

\section{Ramifications and simulations}
\label{sec:ram_and_simu}
\subsection{Algorithm} 
\label{subsec:ram_and_simu_algo} In this section, we describe how our result can be used to approximate the empirical distribution $\emp_t^{G}$ of an IPS at time $t$ on any finite graph $G=(V,E).$ 
Let $\empirdegree^G$ denote its empirical degree distribution: for $k\in \N,$
\begin{equation}
\label{eq:empirical_degree_distribution}
\empirdegree_G(k):= \frac{1}{|G|}\sum_{v\in G} \one_{\{d_v=k\}}, \qquad k\in\natzero.
\end{equation}

Given a finite graph $G$, Corollary \ref{cor:limit_thm_CM} naturally suggests  the following approximation for the empirical distribution  $\emp^G_t$ of the IPS on $G$: approximate the empirical distribution $\emp^G_t(a)$ by $\sum_{\vec{b}\in\cup_{n\in\supoff}\stateS^{n+1}} \one_{\{\vec{b}_0=a\}}\lawMF^{\empirdegree_G}_t({\vec{b}})$, where $\empirdegree_G$  is defined in \eqref{eq:empirical_degree_distribution} and $\lawMF^{\empirdegree_G}$ is the unique solution to the ODE \eqref{eq:MF-ODE-UGW}-\eqref{eq:MF-ODE-UGW_initial}, with $\offspring=\empirdegree_G$ from Proposition \ref{prop:ODE-wellposed}. Note that since $G$ is finite, we automatically have $|\Supp(\empirdegree_G)|<\infty.$

\subsection{Examples}
In this section, we provide three examples of models that fall within the framework of our results.
\label{subsec:ram_and_simu_examples}
\subsubsection{Seizure propagation dynamics}
\label{subsubsec:neural_SIR}
A probabilistic model of seizure-spreading dynamics in the brain was introduced in \cite{MoosTrucc23}. This model can be considered as an SIR-type model with types that take into account the inhibitory or excitatory effect of neurons on the dynamics. Vertices represent neurons in various states with respect to a seizure. The state S corresponds to the state of a neuron before a seizure, I indicates that a seizure is currently spreading through the neuron and R represents a resting state after the seizure. Unlike classical SIR dynamics, the transition rates between states are nonlinear and the transition from state I to R also depends on the states of neighboring neurons. We introduce a Markov idealization of their model that ignores delays in the transition rates. For notational convenience, we identify the states of the nodes (susceptible to propagating the seizure, in a seizure, or recovering from a seizure) with 0,1 and 2 respectively, and allow two possible values of a neuron's excitability, $-\excit_-$ and $\excit_+$.

Let $\excit_-,\excit_+\in\R_+$ and let $\stateS=\{0,1,2\}\times \{-\excit_-,\excit_+\}$ be the state space of the processes $X^{G_n}=(X^{G_n}_v)_{v\in V}$, where $G_n=CM_n(\offspring)$ with degree distribution $\offspring$ satisfying $\Supp(\offspring)<\infty$ and permissible jump space $\Jmc=\{(1,0)\}$, defined as the solution to the following system of SDEs:
\begin{equation}
\label{eq:nSIR_SDE}
    X_v(t)= X_v(0) + \int_{(0,t)\times \R_+}  \vec{j}_1\one_{\{u<\r(s,X_v(s-),X_{\partial_v}(s-))\}} \mathbf{N}_v(\diff s,\diff u),
\end{equation}
where $\vec{j_1}=(1,0)$, that is, each neuron's intrinsic excitability does not change over time, 
and the transition rates are given by, for $v\in \V$ and $s\in\R_+$, $y_v,\excit_v)\in\stateS_\extra$, and $(x,\excit))\in\stateS_\extra^{\dmax}$,
\begin{equation}
    \label{eq_general_simplified_rate}
    \r(s,(y,\excit_v),(x,\excit)):= \begin{cases}
        0 \vee (\excit_v+\sum_{w\in \partial v} (\beta \one_{\{x_w(s-)=1\}} -\excit_w^- \one_{\{x_w(s-)=0\}})) & \text{ if } y(s-)=0,
        \\  0 \vee \frac{1 + \sum_{w\in \partial v }\excit_w^-\one_{\{x_w(s-)=0\}}}{d- \left(1 + \sum_{w\in \partial v}\excit_w^-\one_{\{x_w(s-)=0\}}\right)} & \text{ if } y(s-)=1,
    \end{cases}
\end{equation}
where $\excit_w^-=-\excit_w\vee 0.$
This model falls within the class of IPS dynamics \eqref{eq:IPS_SDE} and satisfies Assumptions \ref{Ass:A_1}, \ref{Ass:B_1}, \ref{Ass:A_2} and \ref{Ass:A_3}. Unlike classical SIR models, it has nonlinear transition rates, wherein neurons with negative excitability inhibit the propagation of the seizure and facilitate recovery from it in their neighborhood.

Figure \ref{fig:nSIR} illustrates the accuracy of this approximation even for graphs of moderate size.
\begin{figure}
    \centering
    \includegraphics[width=\linewidth]{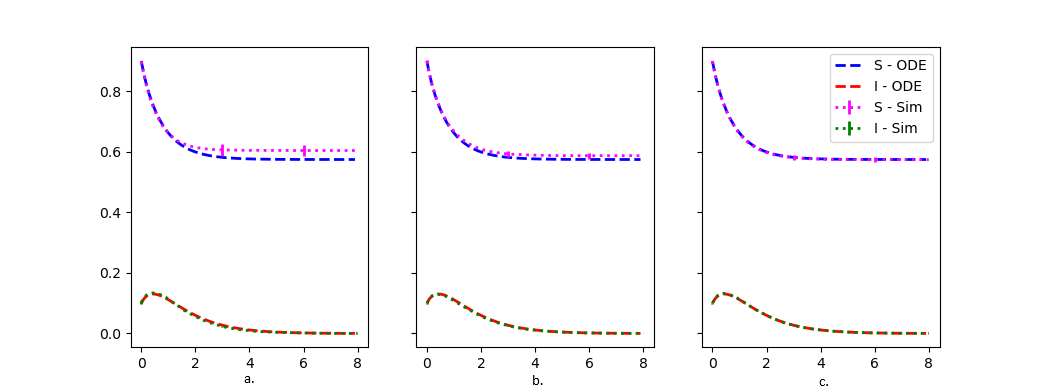}
    \caption{Fraction of individuals in states S and I in the seizure propagation model from Section \ref{subsubsec:neural_SIR} on the random 3-regular graph with the number of vertices equal to $n=50$ (a), $200$ (b), $400$ (c). We compare our ODE (dashed lines) with simulations (dotted lines) averaged over 500 runs.}
    \label{fig:nSIR}
\end{figure}
    
\subsubsection{The entrenched majority voter model}
\label{subsubsec:voter_model}
We introduce here a toy voter model on the UGW tree. We first give an informal description of the model. Each vertex represents an individual in one of three states: undecided (0), blue (1) and red (-1). At random times given by point processes associated with each vertex, undecided voters may change state depending on whether blue or red neighbors have the majority. If there is no majority, the vertex stays undecided. Once a vertex is in state blue or red, it stays there. 

This model can also be thought of as a toy model to represent two competing expanding populations on a sparse graph.

This model has state space $\stateS=\{-1,0,1\}$. We consider the processes $X^{G}=(X^{G}_v)_{v\in V}$, where $G$ is a UGW tree with offspring distribution $\offspring$ satisfying Assumption \ref{Ass:A_1}, defined as the solution to the system of SDEs \eqref{eq:IPS_SDE} with

\begin{equation*}
\r^{1}(s,X_v(s-),X_{\partial_v}(s-))=\one_{\{X_v(s-)=0\}}\one_{\{|\{v \in \partial_v, X_v(s-)=1\}|>|\{v \in \partial_v, X_v(s-)=-1\}|\}}
\end{equation*}
and
\begin{equation*}
\r^{-1}(s,X_v(s-),X_{\partial_v}(s-))=\one_{\{X_v(s-)=0\}}\one_{\{|\{v \in \partial_v, X_v(s-)=1\}|<|\{v \in \partial_v, X_v(s-)=-1\}|\}}.
\end{equation*}

Note that for this model, the mean-field framework, which consists of approximating $G$ by the complete graph, provides a poor approximation due to the disproportionate influence of the initial conditions. Indeed, if the underlying graph is complete, majority decisions always go the way of the initial majority opinion: if at time 0, there are more vertices in state 1 than in state -1, all state changes in the dynamics will be from 0 to 1. Figure \ref{fig:entrenched_majority} illustrates the poor performance of the mean-field regime, while showcasing that the Markov local-field performs well in this setting.

\begin{figure}
    \centering
    \includegraphics[scale=0.5]{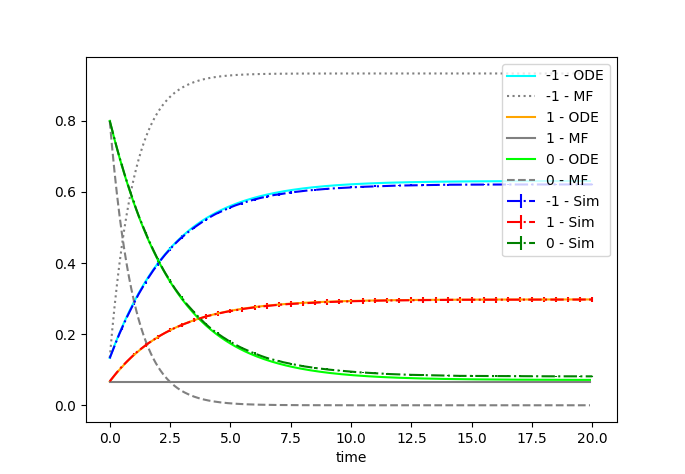}
    \caption{Fraction of individuals in each state of the entrenched majority voter model from Section \ref{subsubsec:voter_model} on the random $2-$regular graph with 200 vertices. We compare our ODE,  simulations (500 realizations), and the mean-field approximation.}
\label{fig:entrenched_majority}
\end{figure}

\subsubsection{Multivariate Markov Hawkes processes with threshold}
\label{subsubsec:Hawkes}
Hawkes processes, originally developed to model earthquakes, have in recent years been extensively used in computational neuroscience to model spiking neurons \cite{ditlevsen_locherbach_17}\cite{ZAN_2022}. We introduce here a modified version of Hawkes processes with a fixed threshold as our main result Theorem \ref{thm:main} requires the state space to be finite.

Let $G=(V,E)$ be a finite graph. An $(\mathcal{F}_t)$-adapted multivariate counting process $(X_v(t))_{t\geq 0, v \in V}$ defined on a filtered probability space $(\Omega,\mathcal{F},(\mathcal{F}_t)_{t\geq 0},\P)$ is hereafter referred to as a multivariate Hawkes process on $G$ with a threshold $M>0$ if $\P$-almost surely, for all $t\geq 0$ and $v\in V,$ 
\begin{equation}
\label{eq:Hawkes_threshold_def}
X_v(t)=M \wedge \left(\int_0^t \int_0^\infty \mathbf{1}_{\{u\leq \rho(s,X_v(s-),X_{\partial_v}(s-))\}} \mathbf{N}_v(\diff s,\diff u)\right),
\end{equation}
where
\begin{equation}
\label{eq:Hawkes_threshold_rate}
\rho(t,X_v(t-),X_{\partial_v}(t-))=f\left(u_v(t)+\sum_{w \in \partial_v} \int_0^{t-} h(t-s) X_w(\diff s)\right),
\end{equation}
with $f,(u_v)$ and $h$ are functions satisfying the following conditions:
\begin{itemize}
    \item the functions $(u_v)_{v \in V}:\R_+\rightarrow\R$ are uniformly bounded in $t$ and $v$.
    \item the function $f:\R\rightarrow\R_+$ is Lipschitz-continuous and either non-negative or equal to identity with $u_v\geq 0$ for all $v\in V$
    \item the function $h:\R_+\rightarrow\R$ is locally square integrable on $[0,+\infty)$;
\end{itemize}
The multivariate Hawkes process is generally non-Markovian. However, it becomes Markov in the special case when for all $t\geq 0,$ $h(t)=\alpha e^{-\beta t},$ with $\alpha,\beta>0$ \cite{Oakes75}.
In this case, $X_v$ defined by \eqref{eq:Hawkes_threshold_def} falls within the IPS framework \eqref{eq:IPS_SDE} with state space $\stateS=\{0,1,\ldots,M\}$ and permissible jump space $\Jmc=\{1\}$, and satisfies Assumptions \ref{Ass:A_1}, \ref{Ass:B_1}, \ref{Ass:A_2} and \ref{Ass:A_3}.
\section{Outline of the proof}
\label{sec:proof_outline}

\subsection{Summary of prior results}
\label{subsec:outline_summary_past_results}

\subsubsection{The local-field equations (LFE)}
\label{subsec:LFE}
We first recall a result from \cite{GanRamUGW} that autonomously characterizes the law of $X_{\V^\offspring_1}$ as the solution of a certain jump SDE called the local-field equation. However, this characterization shows that, in general, $X_{\V^\offspring_1}$ is non-Markovian and extremely challenging to analyze and simulate. Our contribution is to derive a tractable characterization of the time-space marginals $X_{\V^\offspring_1}(t)$, $t\in\R_+$.

\begin{definition}[Local-field Equations]\label{def:lfe}
Let $\Nlf=\{\Nlf_v^j\}_{j\in \Jmc, v\in\V_1^\offspring}$ be a collection of independent Poisson processes on $\R_+\times\R_+$ with 
unit intensity, $\Xlfinitial$ be a random element in $\stateS_\extra^{\V^\offspring_1}$ such that $\Xlfinitial\deq X_{\V^\offspring_1}(0)$, and set $D=\sum_{v\in\V_1^\offspring} \one_{\{\Xlfinitial_v \neq \extra\}}$. We define a stochastic process $\Xlf$ on $\stateS_\extra^{\V^\offspring_1}$ by the following SDE:

    \begin{equation}\label{eq:lf-SDE}
        \Xlf_v(t)= \Xlfinitial_v+\sum_{j \in \Jmc}\int_{(0,t)\times \R_+} j  \one_{\{u<\rlf_v^{j}(s,\Xlf)\}} \mathbf{\Nlf}^j_v(\diff s,\diff u),
    \end{equation}
    where $\rlf_v^j:[0,\infty)\times \Dmc^{\V_1^\offspring}\rightarrow [0,\infty)$ is given by
    \begin{equation}\label{eq:lf-rates}
        \rlf_v^j(t,x):= \begin{cases}
            \r^j(t,x_\root(t-),x_{\partial_\root}(t-)) & \text{ if } v=\root,
            \\   \displaystyle \frac{\E[D\r^j(t,\Xlf_\root(t-),\Xlf_\root(t-))| \Xlf_\root[t)=x_v[t) , \Xlf_1[t)=x_\root[t)] }{\E[D|  \Xlf_\root[t)=x_v[t) , \Xlf_1[t)=x_\root[t)] }& \text{ otherwise } .
        \end{cases}
    \end{equation}
    We refer to equations \eqref{eq:lf-SDE}-\eqref{eq:lf-rates} as the \textit{local-field equation}, henceforth abbreviated to LFE.
\end{definition}
\begin{remark}
    We note that, while $\rmf_v^j$ is defined as a function on $[0,\infty)\times\Dmc^{\V_1^\offspring}$, for any $t\in[0,\infty)$, $\rmf_v^j(t,x)$ only depends on $x[t)$, the trajectory up to time $t$. More precisely, for every $x,y\in\Dmc^{\V_1^\offspring}$ such that $x[t)=y[t)$, we have that $\rmf_v^j(t,x)=\rmf_v^j(t,y)$.
\end{remark}

\begin{theorem}\label{thm:LFE}
    Suppose that Assumptions \ref{Ass:A_1}, \ref{Ass:B_1} and \ref{Ass:A_2} hold. Then, we have
    \begin{equation*} \law(X_{\V_1^\offspring})=\law(\Xlf).
    \end{equation*}
\end{theorem}
\begin{proof}
This is established in \cite{GanRamUGW}. For the special case of a regular tree, see \cite{Gangulythesis}.
\end{proof}

\begin{remark}
Even  when the
dynamics given by \eqref{eq:IPS_SDE} are Markov, the $\Xlf$ is in general a non-Markovian process. Moreover, $\Xlf$ is nonlinear in the sense that the infinitesimal evolution at time $t$ depends on the law $\law(\Xlf[t))$, rather that only on the trajectory $\Xlf[t)$.
\end{remark}
\subsubsection{A trajectorial 2-MRF property}
\label{subsec:traj_2-MRF}
A key ingredient in the derivation of the local-field equations is the following 2-MRF property for trajectories of the IPS:

\begin{theorem}\label{thm:regular-2mrf}
Suppose that  Assumptions \ref{Ass:A_1},\ref{Ass:B_1} and \ref{Ass:A_2} hold. Let $X$ be the solution to the IPS \eqref{eq:IPS_SDE}, and for $t>0$, let $X[t]:=\{X(s), s\in [0,t]\}$ denote the the trajectory of $X$ up to time $t.$ Then for every $t\in [0,\infty),$  $X[t]=(X_v[t])_{v\in\V}$ forms a 2-MRF w.r.t. $G$ in the sense of Definition \ref{def:2-MRF}.
\end{theorem}
\begin{proof}
    This follows from \cite{GanRamMRF}*{Theorem 3.7} upon observing that Assumptions \ref{Ass:A_2} and \ref{Ass:B_1} imply \cite{GanRamMRF}*{Assumptions 3.1 and 3.4}.
\end{proof}

\subsubsection{Unimodularity}
\label{subsec:unimod}
We now give a precise definition of unimodularity, a symmetry property appearing in Assumptions \ref{Ass:A_1} and \ref{Ass:B_1} which transcribes the intuition that a random, possibly infinite, graph "looks the same" from every vertex. We refer to \cite{aldous-lyons} for a more thorough discussion. 

A \emph{rooted graph} $(G,o)$ is a connected graph equipped with a distinguished vertex $o$, where we assume $G$ has finite or countable vertex set and is locally finite, meaning each vertex has finitely many neighbors.
An \emph{isomorphism} from one rooted graph $(G_1,o_1)$ to another $(G_2,o_2)$ is a bijection $\varphi$ from the vertex set of $G_1$ to that of $G_2$ such that $\varphi(o_1)=o_2$ and such that $(u,v)$ is an edge in $G_1$ if and only if $(\varphi(u),\varphi(v))$ is an edge in $G_2$.
We say two rooted graphs  are \emph{isomorphic} if there exists an isomorphism \ between them, and we let $\G_*$  denote the set of isomorphism classes of rooted graphs. Similarly, a \emph{doubly rooted graph} $(G,o,o')$ is a rooted graph $(G,o)$ with an additional distinguished vertex $o'$ (which may equal $o$). Two doubly rooted graphs $(G_i,o_i,o_i')$ are isomorphic if there is an isomorphism from $(G_1,o_1)$ to $(G_2,o_2)$ which also maps $o_1'$ to $o_2'$. We write $\G_{**}$ for the set of isomorphism classes of doubly rooted graphs.

There are analogous definitions for \emph{marked rooted graphs}.
An \emph{$\X$-marked rooted graph} is a tuple $(G,x,o)$, where $(G,o)$ is a rooted graph and $x=(x_v)_{v \in G} \in \X^G$ is a vector of marks, indexed by vertices of $G$.
We say that two marked rooted graphs $(G_1,x^1,o_1)$ and $(G_2,x^2,o_2)$  are \emph{isomorphic} if there exists an isomorphism $\varphi$
between the rooted graphs $(G_1,o_1)$ and $(G_2,o_2)$ that maps the marks of one to the marks of the other (i.e., for which $x^1_{\varphi(v)} = x^2_v$  for all $v \in G$). Let  $\G_*[\X]$ denote the set of isomorphism classes of $\X$-marked rooted graphs.
A double rooted marked graph is defined in the obvious way, and $\G_{**}[\X]$ denotes the set of isomorphism classes of doubly rooted marked graphs.

These spaces of graphs come with natural topologies. For $r \in \N$ and $(G,o)\in\G_*$, let $B_r(G,o)$ denote the induced subgraph of $G$ (rooted at $o$) containing only those vertices with (graph) distance at most $r$ from the root $o$. The distance between $(G_1,o_1)$ and $(G_2,o_2)$ is defined as the value $1/(1+\bar{r})$, where $\bar{r}$ is the supremum over
$r \in \N_0$ such that $B_r(G_1,o_1)$ and $B_r(G_2,o_2)$ are isomorphic,
where we interpret $B_0(G_i,o_i) = \{o_i\}$.  The distance between two marked graphs $(G_i,x^i,o_i)$, $i = 1, 2$, is likewise defined as the value $1/(1+\bar{r})$, where $\bar{r}$ is the supremum over $r \in \N_0$ such that there exists an isomorphism $\varphi$ from $B_r(G_1,o_1)$ to $B_r(G_2,o_2)$ such that $d(x^1_v,x^2_{\varphi(v)}) \le 1/r$ for all $v \in B_r(G_1,o_1)$. We equip $\G_{**}$ and $\G_{**}[\X]$ with similar metrics, just using the union of the balls at the two roots, $B_r(G,o) \cup B_r(G,o')$, in place of the ball around a single root $B_r(G,o)$. Metrized in this manner, the spaces $\G_*$ and $\G_{**}$ are Polish spaces, as are $\G_*[\X]$ and $\G_{**}[\X]$ if $\X$ is itself a Polish space. See \cite{Bordenave2016}*{Lemma 3.4} for a proof
that $\G_*[\X]$ is a Polish space. 
Each space $\G_*[\X]$ and $\G_{**}[\X]$ is equipped with its Borel $\sigma$-algebra.

We are now ready to introduce the definition of unimodularity for general graphs.

\begin{definition}[Unimodular Graph]  \label{def-unimodular}
For a metric space $\X$,  
we say that a $\G_*[\X]$-valued random element $(G,Y,o)$ is \emph{unimodular}
if the following \emph{mass-transport principle} holds: for every (non-negative) bounded Borel measurable function $F : \G_{**}[\X] \rightarrow \R_+$,
\begin{equation}
\E\left[\sum_{o' \in G}F(G,Y,o,o')\right] = \E\left[\sum_{o' \in G}F(G,Y,o',o)\right].   \label{eq:unimodularity}
\end{equation}

\end{definition}

\subsection{A time marginal 2-MRF property}
\label{subsec:proof_outline_2-MRF(t)}

The first ingredient of the proof of Theorem \ref{thm:main} is a conditional independence property that plays a key role in establishing an autonomous characterization of $\law(X_{\V^\offspring_1}(t))$ for all $t\geq 0$.
Namely, we show that the time marginals of the dynamics form a 2-MRF in the sense of Definition \ref{def:2-MRF}. This is a stronger result compared to Theorem \ref{thm:regular-2mrf}.

\begin{theorem}
\label{thm:2MRF_Markov}
    Suppose that Assumptions \ref{Ass:A_1}, \ref{Ass:B_1}, \ref{Ass:A_2} and \ref{Ass:A_3}  hold. Then, for every $t\in [0,\infty)$, $X(t)=(X_v(t))_{v\in\V}$ forms a 2-MRF w.r.t. $G$ in the sense of Definition \ref{def:2-MRF}.
\end{theorem}
The proof of Theorem \ref{thm:2MRF_Markov} is given in Section \ref{subsec:proof_2-MRF(t)_from_[tau]_to_(t)}. It relies on combining a 2-MRF property for trajectories up to stopping times (Proposition \ref{prop:2MRF_pathwise_stoptime}, proved in Section \ref{subsec:proof_2-MRF(t)_2-MRF[tau]}) with structural properties of the graph, analyzed in Section \ref{subsec:proof_2-MRF(t)_from_[tau]_to_(t)}.

\subsection{Markovian projection} 
\label{subsec:proof_outline_mimicking}

At its heart, the tractable description of the IPS comes from a Markov approximation of the local field equations \eqref{eq:lf-SDE}. The goal of this section is to state an application of a Markovian projection result for jump processes to the local-field equation that we will require in order to prove the main theorem. By Markovian projection, we mean here a Markov process with the same marginal laws as the process of interest. The terms mimicking and filtering can also be found in the literature. For consistency purposes, we will refer to mimicking when the Markovian projection has the same marginal laws, and to filtering when a stronger, pathwise equality of laws takes place. This question has accrued considerable interest in the literature, stemming from the work of Gy\"ongy \cite{gyongy1986mimicking} who constructed a Markovian projection for diffusive processes. This work was extended by Brunick and Shreve \cite{BrunickShreve2013} by relaxing assumptions down to integrability conditions. As far as the authors are aware, there are no works in the literature discussing Markovian projections for pure-jump processes. However, some more general results on semimartingales with jumps exist. Bentata and Cont \cite{bentata2009mimicking} prove a Markovian projection result in such a setting at the cost of regularity conditions on the coefficients which are not satisfied in the models we are considering. K\"opfer and R\"uschendorf \cite{kopfer2023markov} provide an alternative construction to Bentata and Cont using assumptions that are also hard to check. Larsson and Long \cite{larsson2024markovian} construct a Markovian projection for semimartingales with jumps in a weaker sense than that of the previous works, namely that the projection is obtained as the solution to a Markovian martingale problem with very mild boundedness and integrability assumptions, but the uniqueness and Markov property of this mimicking process are not guaranteed. In this work, we show that the assumptions for the existence of a solution to the Markovian martingale problem are satisfied in our case. We then derive the uniqueness of the solution by considering the SDE that the martingale problem is equivalent to, and showing by other arguments that this SDE admits a unique solution. In Section \ref{sec:Markov_proj_proof}, we state a general mimicking result for pure jump IPS with potentially trajectory-dependent jump rates.
Here, we state an instanciation of this general mimicking result to derive the Markovian projection of the local-field equations from Definition \ref{def:lfe}.

\begin{theorem}\label{thm:markov_proj}
Let $\Xlf$ be the local-field process given by \eqref{eq:lf-SDE} with rates $(\rlf^j_v)$ given by \eqref{eq:lf-rates}
. Let $\Xmp$ be an IPS on $G=\V^\offspring_1$ with rates $\rmp^j_{v}:[0,\infty)\times \stateS_\extra^{\V_1^\offspring}\rightarrow[0,\infty)$ given by
\begin{equation}\label{eq:markov_proj_rate}
    \rmp^j_{v}(t,x)=\E[\rlf^j_{v}(t,\YLE) | \YLE(t-)=x],
\end{equation} $j\in\Jmc$, $v\in\natzero$, that is, $\Xmp$ is the solution of \eqref{eq:IPS_SDE} with the rates $\r_v^j$ replaced by $\rmp_v^j$. Then for every $t\geq0$,
\begin{equation*}
\Xmp(t)\overset{d}{=} \YLE(t).        
\end{equation*}
\end{theorem}

Theorem \ref{thm:markov_proj}  provides a characterization of the time-marginals of $X_{\V_1^\offspring}$ as the solution of SDE \eqref{def:IPS} with rates $\rmf_v^j$ given by \eqref{eq:markov_proj_rate}. However, the rates $\gamma_v^j(t,\cdot)$ depend on the law of the time-paths $\law(X_{\V_1^\offspring}[t))$, and therefore Theorem \ref{thm:markov_proj} does not provide an autonomous characterization of  $\law(X_{\V_1^\offspring}(t))$.
In the next section, we combine Theorem \ref{thm:markov_proj} and the time-marginal 2-MRF property established in Theorem \ref{thm:2MRF_Markov} to obtain a tractable characterization of the time-marginals of $X^{\Tmc}$ as the solution of an SDE that depends only on $X_{\V_1^\offspring}(t)$ and $\law(X_{\V_1^\offspring}(t))$, and is therefore autonomous.

\subsection{Obtaining the ODE}
\label{subsec:proof_outline_wrapup} 
In this section, we introduce the Markov IPS corresponding to the LFE and show how its forward Kolmogorov equations characterize the law of the space-time marginals of the LFE \eqref{eq:lf-SDE}.

\begin{definition}[Markov local-field equations]\label{def:M-LFE}
Let $\Nmf=\{\Nmf_v^j\}_{j\in \Jmc, v\in\V_1^\offspring}$ be a collection of independent Poisson processes on $\R_+\times\R_+$ with 
unit intensity, $\Xmfinitial$ be a random element in $\stateS_\extra^{\V^\offspring_1}$ such that $\Xmfinitial\deq X_{\V^\offspring_1}(0)$, and set $D=\sum_{v\in\V^\offspring_1} \one_{\{\Xmfinitial_v \neq \extra\}}$. We define a stochastic process $\Xmf$ on $\stateS_\extra^{\V^\offspring_1}$ by the following SDE:

    \begin{equation}\label{eq:MLF-SDE}
        \Xmf_v(t)= \Xmfinitial_v+\sum_{j \in \Jmc}\int_{(0,t)\times \R_+} j  \one_{\{u<\rmf_v^{j}(s,\Xmf(s-)\}} \mathbf{\Nlf}^j_v(\diff s,\diff u),
    \end{equation}
    with
    \begin{equation}\label{eq:mf-rate}
        \rmf_v^j(t,y):= \begin{cases}
            \r^j(t,y_\root,y_{\partial_\root}) & \text{ if } v=\root,
            \\ \displaystyle \frac{\E[D\r^j(t,\Xmf_\root(t-),\Xmf_{\partial_\root}(t-))| \Xmf_\root(t-)=y_v , \Xmf_1(t-)=y_\root) ]}{\E[D|  \Xmf_\root(t-)=y_v , \Xmf_1(t-)=y_\root) ] }& \text{ otherwise}.
        \end{cases}
    \end{equation}
     We refer to equations \eqref{eq:MLF-SDE}-\eqref{eq:mf-rate} as the \textit{Markov local-field equation}, or MLFE for short.
\end{definition}
\begin{proposition}\label{prop:MLFE-wellposed}
    Under Assumptions \ref{Ass:A_1}, \ref{Ass:B_1} and \ref{Ass:A_2}, the Markov local-field equation \eqref{eq:MLF-SDE} is well-posed.
\end{proposition}
\begin{proof}
The well-posedness stems from considerations identical to the well-posedness of the local-field equation from Definition \ref{def:lfe}, for which we refer to \cite{GanRamUGW}. 
\end{proof}

We now conclude this section with the proof of the main result, Theorem \ref{thm:main}.
\begin{proof}[Proof of Theorem \ref{thm:main}] 
By Theorem \ref{thm:LFE}, $\law(X_{\V_1^\offspring})=\law(\Xlf)$, where $\Xlf$ is the solution of the local-field equations \eqref{eq:lf-SDE}. Therefore, characterizing the time marginals of $X_{\V_1^\offspring}$  is equivalent to characterizing the time marginals of $\Xlf$. 
By Theorem \ref{thm:markov_proj}, there exists a Markov  process $\Xmp$ such that $\law(\Xmp(t))=\law(\Xlf(t))$ for all $t\geq 0$, and $\Xmp$ solves the SDE \eqref{eq:Markov_proj_SDE} with rates 
\begin{equation}
    \label{eq:rmp}
    \rmp^j_{v}(t,x)=\E[\rlf^j_{v}(t,\YLE) | \YLE(t-)=x].
\end{equation}

We next use the 2-MRF property proved in Theorem \ref{thm:2MRF_Markov} to show that $\law(\Xmp)=\law(\Xmf)$, where $\Xmf$ is the solution to the Markov local-field defined in \eqref{eq:MLF-SDE}. By definition, the initial states $\Xmf(0)$ and $\Xmp(0)$ have the same law. By well-posedness of the Markov local-field equations given in Proposition \ref{prop:MLFE-wellposed}, it is enough to show that the rates $\rmp=\{\rmp_v^j\}_{v\in \V_1^\offspring, j\in\Jmc}$ and $\rmf=\{\rmf_v^j\}_{v\in \V_1^\offspring, j\in\Jmc}$ coincide, where $\rmf$ are the rates of $\Xmf$ defined in \eqref{eq:mf-rate}. For all $j\in\Jmc$ and $x\in\stateS_\extra^{\V_1^\offspring}$, by definition of the rates $\rlf_\root^j$ given in \eqref{eq:lf-rates},

\begin{equation*}
     \rmp^j_\root(t,x)=\E[\rlf_\root^{j}(t,\YLE) \ |\  \YLE(t-)=x]=
         \E[\r^j(t,\YLE_\root(t-), \YLE_{\partial_\root}(t-)) | \YLE(t-)=x].
\end{equation*}
Therefore, using \eqref{eq:mf-rate} in the second equality, we have
     \begin{align}
        \begin{split}
            \E[\r^j(t,\YLE_\root(t-), \YLE_{\partial_\root}(t-)) | \YLE(t-)=x]
             &= \r^j(t,x_\root, x_{\partial_\root}) \\& =\rmf_\root^{j}(t,x),
        \end{split}
    \end{align}
and thus $\rmp_\root^j(t,x)=\rmf_\root^{j}(t,x)$.

Now consider the neighbors of the root, and  let $v\in \partial_\root$. Since by Theorem \ref{thm:LFE}, $\law(\Xlf)=\law(X_{\V_1^\offspring})$, recalling the definition \eqref{eq:lf-rates} of $\rlf$,  for $y\in\Dmc^{\V_1^\offspring}$, we have
\begin{align*}
    \rlf_v^j(t,y) &= 
     \frac{\E[D\r^j(t,\Xlf_{\root}(t-),\Xlf_{\partial_\root}(t-))| \Xlf_\root[t)=y_v[t) , \Xlf_1[t)=x_\root[t)] }{\E[D|  \Xlf_\root[t)=y_v[t) , \Xlf_1[t)=y_\root[t)] }
     \\ &= \frac{\E[D\r^j(t,X_{\root},X_{\partial_\root}(t-)| X_\root[t)=y_v[t) , X_1[t)=y_\root[t)] }{\E[D|  X_\root[t)=y_v[t) , \Xlf_1[t)=y_\root[t)] }.
\end{align*}
Using unimodularity, analogously to the proof of \cite{lacker2023marginal}*{Proposition 3.18},
it follows that 
\begin{equation}\label{eq:rate_after_unimod}
    \rlf_v^j(t,y)=\E[\r^j(t,X_{v}(t-),X_{\partial_v}(t-)| X_\root[t)=y_\root[t) , X_v[t)=y_v[t)].
\end{equation}
Using the fact that $\law(\Xlf)=\law(X_{\V_1^\offspring})$ and combining \eqref{eq:rate_after_unimod} with \eqref{eq:rmp}, we have 
\begin{align}\label{eq:main_proof_1}
    \begin{split}
        \rmp_v^j(t,y(t-)) &= \E[\rlf^j_{v}(t,\YLE)\ |\ \YLE(t-)= y(t-)] 
        \\& = \E[\rlf^j_{v}(t,X_{\cl_\root})\ |\ X_{\cl_\root}(t-)=y(t-)].  
        \\& = \E [ \E[\r^j(t,X_{v}(t-),X_{\partial_v}(t-)| X_\root[t)=y_\root[t) , X_v[t)=y_v[t)] |\ X_{\cl_\root}(t-)= y(t-)].
    \end{split}
\end{align}

We now proceed to use the 2-MRF properties for trajectories (Theorem \ref{thm:regular-2mrf}) and time marginals (Theorem \ref{thm:2MRF_Markov}) to simplify the last equation. For that purpose, let $A\subset \V$ be obtained from $\V$ by removing the root and the subtree rooted at $v$, that is, $A=\V \setminus (\{\root\}\cup\{vw : w\in \V\}$, where we recall that $vw$ denotes concatenation with the convention $v\root=v$. Note that $\partial^2 A=\{v,\root\}$. Also, let $B=\partial_\root\setminus\{v\}$, which satisfies $B\subset \V\setminus A$.
By Theorem \ref{thm:regular-2mrf},
\begin{equation}\label{eq:main_proof_2}
    \E[\r^j(t,X_{v}(t-),X_{\partial_v}(t-)| X_{\root,v}[t)]=\E[\r^j(t,X_{v}(t-),X_{\partial_v}(t-)| X_{\cl_\root}[t)].
\end{equation}
Combining \eqref{eq:main_proof_1} and \eqref{eq:main_proof_2}, and applying the tower property of conditional expectations, 

\begin{align}
    \begin{split}
    \rmp_v^j(t,y(t-))=&\E [ \E[\r^j(t,X_{v}(t-),X_{\partial_v}(t-)| X_{\cl_\root}[t)=y[t)] |\ X_{\cl_\root}(t-)= y(t-)] 
        \\=& \E[ \r^j(t,X_{v}(t-),X_{\partial_v}(t-) | \ X_{\cl_\root}(t-)= y(t-)] 
    \end{split}
\end{align}

Finally, we invoke the 2-MRF property for time marginals (Theorem \ref{thm:2MRF_Markov}) with the same choice of $A$ and $B$ as above to conclude that
    \begin{equation*}
    \E[ \r^j(t,X_v(t-),X_{\partial_v}(t-)) \ |\ X_{\cl_\root}(t-)=y(t-)]  = \E[ \r^j(t,X_v(t-),X_{\partial_v}(t-)) \ |\ X_{\root,v}(t-)=y_{\root,v}(t-)].
    \end{equation*}
The latter is equal to $\rmf_v^j(t,y(t-))$ by definition, see \eqref{eq:mf-rate}, and therefore, $\rmp_v^j(t,x)=\rmf_v^j(t,x)$ for all $t\in[0,\infty)$ and $x\in\stateS_\extra^{\V_1^\offspring}$. This establishes that for every $t\in[0,\infty)$, $\law(X_{\V_1^\offspring}(t))=\law(\Xmf(t))$. 
Equations \eqref{eq:MF-ODE-UGW}-\eqref{eq:MF-ODE-UGW_initial} are a transcription of the forward Kolmogorov equations of $\Xmf$.
\end{proof}

\begin{remark}
    For a full derivation of the forward Kolmogorov equations of the Markov local-field process $\Xmf$ in the special case of $\offspring=\delta_k$ for $k\in\natzero\setminus\{0,1\}$, where the UGW coincides with the $k$-regular tree, we refer the reader to \cite{Gangulythesis}*{Section 10.2}.
\end{remark}

\begin{remark}
\label{rem_counterexample}
Note that Theorem \ref{thm:main} is a result about the marginal laws of the dynamics: it states that the space-time marginals of the LFE \eqref{eq:lf-SDE} and the MLFE \eqref{eq:MLF-SDE} coincide. A natural question to ask is then the validity of a stronger result, namely the equality of the rates in the two SDEs. The answer to that question in all generality is negative. Consider the seizure propagation dynamics from Example \ref{subsubsec:neural_SIR} on the graph $G=\Z$. Using the notation introduced therein,
consider the case where $\{\Xmf_{\root,1}(t)=(1,1)\}$. Let $\tau^1_{\root,1}=\inf \{s\geq 0, \Xmf_{\root,1}(s)=1\}.$ Suppose $\tau^1_\root<\tau^1_1.$ In that case, $\Xmf_{-1}(\tau^1_\root)=1$, since by monotonicity, $\Xmf_{1}(\tau^1_\root)=0.$ Therefore, $\P(\Xmf_{-1}(t)=0)=0.$ This in turn implies that $\E[\Xmf_{-1}(t)|\Xmf_{\root,1}(t)=(1,1)]\neq \E[\Xmf_{-1}(t)|\Xmf_{\root,1}[t]=(1,1)].$
\end{remark}

\section{Proof of the 2-MRF property for time marginals}
\label{sec:proof_2-MRF(t)}
\subsection{Pathwise 2-MRF property up to stopping times}
\label{subsec:proof_2-MRF(t)_2-MRF[tau]}
In this section we establish a  second-order MRF property for the trajectories of $X$ up to a class of stopping times. We start by recalling some elementary definitions.

\newcommand{\sigmaalgebra}{\Smc}
\begin{definition}\label{def:stopping_time}
Let $(\Omega,\sigmaalgebra,\P)$ be a probability space. 
\begin{enumerate}
    \item A random  variable $\tau: \Omega \rightarrow [0,\infty)$ is said to be a \textit{stopping time} with respect to a filtration $\Fmc=\{\Fmc_t\}_{t\geq 0}$ or equivalently an $\Fmc$-stopping time, whenever $\{\tau\leq t\}\in\Fmc_t$ for all $t\geq 0$.
    \item Given a filtration $\Fmc=\{\Fmc_t\}_{t\geq 0}$ on $(\Omega,\sigmaalgebra,\P)$, the $\sigma$-algebra generated by an $\Fmc$-stopping time $\tau$ is the set \begin{equation}\label{eq:F_tau}
        \Fmc_\tau:=\{A\in\sigmaalgebra \ : A\cap\{\tau\leq t\}\in\Fmc_t, \forall\ t\geq 0\}.
    \end{equation}
    \item Let $U=\{U(t)\}_{t\geq 0}$ be a stochastic process on $(\Omega,\sigmaalgebra,\P)$ that takes values in a measurable space. The \textit{natural filtration} generated by $U$, denoted by $\Fmc^U$, is given by \begin{equation}\label{eq:nat_filtration}
        \Fmc^U_t:=\sigma\{U(s) : s\in[0,t]\}, \quad t\geq0.
    \end{equation}
\end{enumerate}
    
\end{definition}

We now state the main result of the section, which generalizes the 2-MRF  up to deterministic times established in \cite{GanRamMRF}, see Theorem \ref{thm:regular-2mrf}.
\begin{proposition}[2-MRF property for stopped trajectories]
\label{prop:2MRF_pathwise_stoptime}
Suppose Assumptions \ref{Ass:A_1}, \ref{Ass:B_1} and \ref{Ass:A_2} hold. Fix $A\subset \V$ with $|\partial^2 A| < \infty$, and let $\tau_A$ be a stopping time with respect to the natural filtration generated by $X_{\partial^2 A}$. Then for every $B\subset \V\setminus (A \cup \partial^2 A)$,
    \begin{equation}
    \label{eq:2MRF_pathwise_stoptime}
        X_A[\tau_A] \perp X_B[\tau_A] | X_{\partial^2 A}[\tau_A].
    \end{equation}
\end{proposition}

The proof of Proposition \ref{prop:2MRF_pathwise_stoptime} is given at the end of this section. As we will show, it is a corollary of a more general result, Theorem \ref{thm:tau_cond_ind}, established below.
To state the latter, we will need a few definitions and basic lemmas. Before we proceed, let us demonstrate how \eqref{eq:2MRF_pathwise_stoptime} can fail when $\tau_A$ is not a stopping time with respect to the natural filtration generated by $X_{\partial^2 A}$.
\begin{example}
    Let $A\subset \V$ be given by $A=\{1v : v\in \V\setminus\{\root\}\}$, where we recall that $wv$ denotes concatenation. Then $\partial^2A=\{1,\root\}$ and $2\not\in A\cup \partial^2 A$. Let $\stateS=\{0,1,2\}$ and $\Jmc=\{1\}$. Define $\tau:=\inf\{t\geq 0: X_{11}(t)=X_2(t)\}$, with the convention that the infimum of an empty set is $\infty$. Then $\tau$ is a stopping time with respect to the filtration generated by $X$, but not a stopping time with respect to the filtration generated by $X_{\partial^2A}$, since $11\not\in \partial^2A$. Moreover, $2\not\in A\cup \partial^2 A$ and, at time $\tau$, $X_{11}(\tau)=X_{2}(\tau)$, and therefore, setting $B=\{2\}$, $X_A(\tau) \not \perp X_{B}(\tau) | X_{\partial_2}[\tau]$.
\end{example}

We now proceed with the definition of a regular conditional distribution.

\begin{definition}[Regular conditional distribution]\label{def:reg_cond}
   Let $(\Omega,\sigmaalgebra,\P)$ be a probability space, $\Hmc\subset\sigmaalgebra$ a sigma-algebra, and $Y$ a random element that takes values in some measurable space $ (\Ymc,\sigmaalgebra^\Ymc)$. A \textit{regular conditional distribution} of $Y$ given $\Hmc$ is a measurable function $f:\sigmaalgebra^\Ymc\times \Omega\rightarrow[0,1]$ that satisfies the following properties:
   \begin{enumerate}
       \item For every $A\in\sigmaalgebra^\Ymc$, $f(A,\omega)=\E[\one_{\{Y\in A\}} | \Hmc](\omega)$ almost surely.
       \item For almost all $\omega\in\Omega$, $A\rightarrow f( A ,\omega)$ is a probability measure on $(\Ymc, \sigmaalgebra^\Ymc)$.
   \end{enumerate}
\end{definition}
\begin{remark}
Although the existence of a regular conditional distribution is not immediate for general measurable spaces, it is known when $\Ymc$ is a Polish space and $\sigmaalgebra^\Ymc$ is its Borel sigma algebra. See \cite{durrett2019probability}*{Theorem 4.1.17})
\end{remark}
Next, we state a useful result from \cite{shiryaev2011optimal}.

\begin{lemma}[\cite{shiryaev2011optimal}*{Theorem 6}] \label{lem:stopped-sigma-algebra}
    Let $(\Omega,\Fmc,\P)$ be a probability space, and let $U$ be a stochastic process that takes values in a measurable space $(\Umc,\sigmaalgebra^{\Umc})$. Let $\Fmc^U$ be the natural filtration generated by $U$, as defined in \eqref{eq:nat_filtration}. Also, suppose that for all $t\geq 0$ and $\omega\in\Omega$, there exists $\omega'\in\Omega$ such that 
    \begin{equation}\label{eq:space-richness}
        U(s)(\omega')=U(s\wedge t)( \omega)\ \forall s\geq 0.   
    \end{equation}
     Then given any $\Fmc^U$ stopping time $\tau$, its stopped filtration is given by  
    \begin{equation*}
    \Fmc_\tau=\sigma( U(s\wedge \tau), s\geq 0).
    \end{equation*}
\end{lemma}

We observe that \eqref{eq:space-richness} holds when $\Omega$ is the càdlàg space $\Dmc(\R_+:U)$ and the process $U$ is given by the canonical process $U(\omega)(s)=\omega(s)$, $\forall s\geq 0$, and hence almost surely càdlàg.

We now establish a key identity.
\begin{lemma} \label{lem:main}
    Let $(\Omega, \sigmaalgebra, \P)$ be a probability space, let  $(\Umc,\sigmaalgebra^{\Umc})$ be a measurable space,  and $\Ymc$ a Polish space equipped with the Borel $\sigma$-algebra $\mathfrak{B}(\Ymc)$. Suppose that $U=\{U(t)\}_{t\in\R_+}$ and  $Y=\{Y(t)\}_{t\in\R_+}$ are continuous time stochastic processes with state space $\Umc$ and $\Ymc$ respectively. Assume that, for all $t\geq 0$ and $\omega\in\Omega$, there exists $\omega'\in\Omega$ such that
    \begin{equation}\label{eq:lemma_condition}
        U(s)(\omega')=U(s\wedge t)( \omega) \ \forall s\geq 0.
    \end{equation}
    Moreover, for each $t\in[0,\infty)$,  let $f_t:\mathfrak{B}(\Ymc) \times \Omega $ be a version of the regular conditional distribution of $Y(t)$ given $\Fmc_t^U$, an element of $\Fmc^U$, the natural filtration generated by $U$. Suppose that $\tau$ is an almost surely finite $\Fmc^U$-stopping time, and let $\Fmc_\tau$ be the sigma algebra generated by $\tau$. Then almost surely,  
    \begin{equation}
        \P( Y(\tau) \in A | \Fmc_\tau)(\omega)=f_{\tau(\omega)}(A,\omega) \qquad \forall A\in\mathfrak{B}(\Ymc).
    \end{equation}
\end{lemma}
\begin{proof}
Since $Y$ has a Polish state space, the regular conditional distribution $f_t$ exists.

For $\omega\in\Omega$ and $A\in\mathfrak{B}(\Ymc)$, define 
\begin{equation}
    g_A(\omega):=\begin{cases}
        f_{\tau(\omega)}(A,\omega), & \text{if } \tau(\omega)<\infty,
        \\ 0, & \text{otherwise.}
    \end{cases}
\end{equation}
The proof follows in two steps. 

Step 1: we prove that $\Omega\ni \omega \rightarrow g_A(\omega)\in [0,1]$ is $\Fmc_\tau / \mathfrak{B}([0,1])$-measurable. We first claim the following: for any $\omega_1,$ $\omega_2\in\Omega$, if
\begin{equation}\label{eq:claim1}
    U(s)(\omega_1)=U(s)(\omega_2) \text{ on } [0,\tau(\omega_1), 
\end{equation}
then 
\begin{equation}\label{eq:claim2}
   \tau(\omega_1)=\tau(\omega_2).
\end{equation}
To see why the claim holds, since $\tau$ is a $\Fmc^U$-stopping time, by Definition \ref{def:stopping_time}(1) and Definition \ref{def:stopping_time}(3), for any $0\leq c \leq \bar{c} <\infty$, there exists a measurable function $H_c:\Dmc([0,\bar{c}]:\Umc)\rightarrow \R$ such that
\begin{equation}\label{eq:claim3}
    \one_{\{\tau(\omega)\leq c\}}= H_c(U(s)(\omega), \forall s \leq \bar{c}), \omega\in\Omega.
\end{equation}
If \eqref{eq:claim1} holds, then applying \eqref{eq:claim3} with $\bar{c}=\tau(\omega_1)$, for any $c\leq \bar{c}$, the right-hand side (and hence, the left-hand side) of \eqref{eq:claim3} is identical for $\omega=\omega_1$ and $\omega=\omega_2$. This implies \eqref{eq:claim2} and proves the claim. 

Since $\Omega\ni\omega \mapsto f_t(A,\omega)$ is $\sigmaalgebra/\mathfrak{B}([0,1])$ measurable for each $t\geq0$ and $\tau$ is a $\Fmc^U$-stopping time, $g$ is $\sigmaalgebra/\mathfrak{B}([0,1])$-measurable. If \eqref{eq:claim1} holds, then by \eqref{eq:claim2} and the fact that $\omega\rightarrow f_t(A,\omega)$ is $\Fmc_t^U$-measurable, it follows that $f_{\tau(\omega_1)}(A,\omega_1)=f_{\tau(\omega_2)}(A,\omega_2)$. Therefore, $g$ is an $\sigmaalgebra/\mathfrak{B}([0,1])$-measurable function that only depends on $\{U(s): s\leq \tau\}$ and  therefore it is $\sigma\{U(s\wedge\tau):\ s\geq 0 \}$-measurable. By Lemma \ref{lem:stopped-sigma-algebra}, $\sigma\{U(s\wedge\tau):\ s\geq 0 \}=\Fmc_\tau$ and therefore Step 1 is complete.

Step 2: Now, let $B\in \Fmc_\tau$. By definition of conditional expectation, it follows that
\begin{equation*}
\int_B \one_{\{Y(\tau(\omega),\omega)\in A\}} \P(\diff\omega)= \int_B \E[ \one_{\{Y(\tau(\omega),\omega)\in A\}} |\ \Fmc_\tau]\P(\diff\omega). 
\end{equation*}
Invoking first Lemma \ref{lem:stopped-sigma-algebra}, and then Definition \ref{def:reg_cond},
\begin{align*}
    \begin{split}
    \int_B \E[ \one_{\{Y(\tau(\omega),\omega)\in A\}} |\ \Fmc_\tau]\P(\diff\omega) &=\int_B \E[ \one_{\{Y(\tau(\omega),\omega)\in A\}} |\ U(s\wedge \tau),\ s\geq 0]\P(\diff\omega)
    \\ &= \int_B f_{\tau(\omega)}(A,\omega)\P(\diff\omega).
    \end{split}
\end{align*}

It follows that $g_A(\omega)=f_{\tau(\omega)}(A,\omega)$ is a version of the conditional expectation $\E[\one_{\{Y(\tau)\}}|\Fmc_\tau]$, and therefore $g_A(\omega)$ is a regular conditional distribution of $Y(\tau)$ given $\Fmc_\tau$.  Therefore, 
\begin{equation}
    \P( Y(\tau) \in A | \Fmc_\tau)(\omega)=f_{\tau(\omega)}(A,\omega) \ \forall  A\in\mathfrak{B}(\Ymc) \quad \text{a.s.,}
\end{equation}
which concludes the proof.

\end{proof}

\begin{theorem}
    \label{thm:tau_cond_ind}
    Let $\Umc$, $\Ymc$ and $\Wmc$ be Polish spaces equipped with their Borel sigma algebras and let $U$, $Y$ and $W$ be continuous time càdlàg stochastic processes with state spaces $\Umc$, $\Ymc$, and $\Wmc$ respectively define on a common probability space $(\Omega,\Fmc,\P)$.   If
    \begin{equation}\label{eq:assumpmrf}
        Y[t)\perp W[t)\ | U[t) \qquad \forall t\in[0,\infty),
    \end{equation}
    then for any a.s. finite $\Fmc^U$-stopping time $\tau$ 
    we have
    \begin{equation}
        Y[\tau) \perp W[\tau)\ | U[\tau) \qquad \text{a.s.}.
    \end{equation}
\end{theorem}
\begin{proof} \draftnote{From Kavita's last feedback, the end of this proof was not clear, but I (June) could not fully follow the handwritten annotations - JC}
    Let $\phi:\Ymc\rightarrow\R$ be measurable. For each $t\in(0,\infty)$, $z\in \Dmc([0,\infty):\Umc)$, and $w\in\Dmc([0,\infty):\Wmc)$, we define
    \begin{equation*}
    g_t(z,w):=\E[f(Y[t))| U[t)=z[t),W[t)=w[t)]
    \end{equation*}
    and 
    \begin{equation*}
    \gbar_t(z):=\E[f(Y[t))| U[t)=z[t)],
    \end{equation*} 
    where we recall the notation $z[t):=\{z(s), s\in[0,t)\}$.
By \eqref{eq:assumpmrf},  $g_t(U,W)=\gbar_t(U)$ almost surely. 

Let $B\in\sigma(U[\tau),W[\tau))$. By the definition of conditional expectation,  it follows that
\begin{equation*}
\E[\one_B f(Y[\tau))] = \E[\one_B\E[f(Y[\tau)) | U[\tau), W[\tau)]].
\end{equation*}
By Lemma \ref{lem:main}, where \eqref{eq:lemma_condition} is satisfied since the processes $U$ and $W$ are càdlàg,  and the definition of $g_t$, we have
\begin{equation*}
\E[\one_B\E[f(Y[\tau)) | U[\tau), W[\tau)]]=\E[\one_B g_\tau( U, W) ].
\end{equation*}
From \eqref{eq:assumpmrf} and the definition of $\gbar_t$, it follows that
\begin{equation*}
\E[\one_B g_\tau( U, W) ]=\E[\one_B \gbar_\tau( U)].
\end{equation*}
Using Lemma \ref{lem:main} a second time, we get
\begin{equation*}
\E[\one_B \gbar_\tau( U)]=\E[\one_B\E[f(Y[\tau)) | U[\tau)] ],
\end{equation*}
which implies that $\E[f(Y[\tau)) | U[\tau)]=\E[f(Y[\tau)) | U[\tau), W[\tau)]$, thus establishing the lemma.
\end{proof}

We now derive Proposition \ref{prop:2MRF_pathwise_stoptime} from Theorem \ref{thm:tau_cond_ind}:

\begin{proof}[Proof of Proposition \ref{prop:2MRF_pathwise_stoptime} ]
Fix $A\subset \V$ with $|\partial^2A|<\infty$.  Let $U:=X_{\partial^2A}$, $Y:=X_{A}$, and $W:=X_{\V\setminus (A\cup\partial^2 A)}$. These are stochastic processes with state spaces $\stateS_\extra^{\partial^2A}$, $\stateS_\extra^A$, and $\stateS_\extra^{\V\setminus (A\cup\partial^2A)}$ respectively, which are marked rooted graphs and therefore Polish. By Theorem \ref{thm:regular-2mrf}, $Y[t) \perp W[t)\ | Z[t) $ for all $t>0$. 
 Theorem \ref{thm:tau_cond_ind} then completes the proof.
    
\end{proof}
\subsection{From the trajectorial to the time-marginal 2-MRF}
\label{subsec:proof_2-MRF(t)_from_[tau]_to_(t)}

This section is devoted to the proof of Theorem \ref{thm:2MRF_Markov}. In order to establish this result, we augment the state space to keep track of the previous state that the vertices are in. This motivates the definition of the following process.

\begin{definition}[Augmented process]
\label{def:y-process}
Let $X$ be the IPS defined by \eqref{eq:IPS_SDE}. For $v \in V$ and $t\geq 0,$ let $X_v^-(t)$ be the last state that $X_v$ was in before $t,$ that is, $X_v^-(t)=X_v(\sigma_v(t)),$ where
\begin{equation*}
    \sigma_v(t):=\sup\{s<t:X_v(s)\neq X_v(t)\},
\end{equation*}
with the convention that $\sigma_v(t):=0$ if the set over which the supremum is empty, and $0-:=0$. 
The augmentation $Y=(Y_v)_{v\in \V}$ of $X$ is the stochastic process with state space $\stateS_\extra^\V\times\stateS_\extra^\V$ defined by 
\begin{equation}
\label{eq:augmented_process}
    Y_v(t):=(X_v(t),X_v^-(t)).
\end{equation} 
\end{definition}
We also define a stopped version of the augmented process:
\begin{definition}[Stopped augmented process]
\label{def:stopped_y_process}
Let $A\subset \V$ be such that $|\partial^2 A| < \infty$, and define
Let 
\begin{equation}
\label{eq:tau_1}
\tauone=\inf \{t>0, X_{\partial^2 A}(t)\neq X_{\partial^2 A}(0)\} 
\end{equation}
be the first jump on the double boundary.
We then define the stopped process 
\begin{equation}
\label{eq:stopped_process}
Y^{\tauone}_v(t):=Y_v(t)\one_{\{t\leq \tauone \}}.
\end{equation}
\end{definition}

Note that defined in that way, $Y$ and $Y^{\tauone}$ are pure jump c\`adl\`ag processes.
We now state two lemmas related to conditional independence relations satisfied by the augmented processes.

\begin{lemma}[Trajectorial 2-MRF property for the augmented process]
\label{lem:pathwise_y_2-MRF}
Under Assumptions \ref{Ass:A_1}, \ref{Ass:B_1} and  \ref{Ass:A_2}, $A\subset \V$ with $|\partial^2 A| < \infty$, and $B\subset \V\setminus (A \cup \partial^2 A)$, we have
    \begin{equation}
    \label{eq:2MRF_pathwise}
        Y^{\tauone}_A[t] \perp Y^{\tauone}_B[t] | Y^{\tauone}_{\partial^2 A}[t], t\in[0,\infty).
    \end{equation}
\end{lemma}
\begin{proof}
For $t> \tauone$, $Y_C^{\tauone}[t]$ is the identically zero trajectory for $C=A,B,\partial^2A$ and so \eqref{eq:2MRF_pathwise} holds trivially. For $t\leq\tauone$, the process $Y^{\tauone}$ coincides with the process $Y$ which is a jump process satisfying the assumptions of Proposition \ref{prop:2MRF_pathwise_stoptime}. Therefore, 
\begin{equation*}
Y_A[t] \perp Y_B[t] | Y_{\partial^2 A}[t], t\in[0,\tauone].
\end{equation*}

\end{proof}

From this pathwise property, we now deduce the following marginal property for the process $Y^\tau$:

\begin{lemma}[Time-marginal 2-MRF property for the augmented process]
\label{lem:marginal_z_2-MRF}
Under Assumptions \ref{Ass:A_1}, \ref{Ass:B_1} and  \ref{Ass:A_2}, $A\subset \V$ with $|\partial^2 A| < \infty$, and $B\subset \V\setminus (A \cup \partial^2 A)$,
    \begin{equation}
    \label{eq:2MRF_marginal_z}
        Y^{\tauone}_A(t) \perp Y^{\tauone}_B(t) | Y^{\tauone}_{\partial^2 A}(t), t\in[0,\infty).
    \end{equation}
\end{lemma}
\begin{proof}
Again, we can restrict ourselves to $t\leq\tauone$. Then by the definition of $\tauone$ given in \eqref{eq:tau_1}, conditioning on $Y^{\tauone}_{\partial^2 A}(t)$ is the same as conditioning on the full trajectory $Y^{\tauone}_{\partial^2 A}[t]$ as there are no jumps prior to $\tauone$ on the double boundary. Therefore, we can apply Lemma \ref{lem:pathwise_y_2-MRF} and replace $Y^{\tauone}_{\partial^2 A}[t]$ by $Y^{\tauone}_{\partial^2 A}(t)$ in \eqref{eq:2MRF_pathwise} to obtain
    \begin{equation*}
        Y^{\tauone}_A[t] \perp Y^{\tauone}_B[t] | Y^{\tauone}_{\partial^2 A}(t),
    \end{equation*}
from which the result immediately follows.
\end{proof}

A final lemma we will require states that our dynamical system does not allow jumps to occur simultaneously on the double boundary of a potentially infinite vertex set and inside it.

\begin{lemma}
\label{lem:simplicity_PP}
Let $A\subset \V$ with $|\partial^2 A| < \infty$.
Let $\tauone$ be defined as in \eqref{eq:tau_1}.
Then,
\begin{equation*}
Y^\tau_A(\tauone-)=Y^\tau_A(\tauone).
\end{equation*}

\end{lemma}
\begin{proof}
Let us reason by contradiction. Suppose that there exists $v\in A$ such that, with positive probability, $X_v(\tauone-)\neq X_v(\tauone),$ that is, there is a jump at time $\tauone$ in $v.$
Note that by Assumption \ref{Ass:A_1}, $G$ is a UGW tree and therefore belongs to the class of spatially localizable graphs, see \cite{GanRam2024}*{Definition 5.1}, \cite{GanRam2024}*{Proposition 5.14}, and \cite{GanRam2024}*{Proposition 5.17}. Therefore, there exists $l\in \N$ such that, denoting by $B_l(v)$ the ball of radius $l$ centered in $v$,  there exists a collection of i.i.d. Poisson point processes $(N_u)_{u \in B_l(v)}$ such that the jumps of $v$ are entirely determined by $(N_u)$ almost surely.
As $|\partial^2 A| < \infty$, the jumps of $v$ and the vertices in $\partial^2 A$ are determined by $(N_u)_{u \in B_l(v)\cup\partial^2 A}$, which form a finite collection of independent Poisson point processes, and therefore are simple point processes. This contradicts the existence of a vertex $v \in A$ with a jump at time $\tauone$ in $v.$
\end{proof}

We now proceed to prove Theorem \ref{thm:2MRF_Markov}.

\begin{proof}[Proof of Theorem \ref{thm:2MRF_Markov}]
Let $A\subset \V$ with $|\partial^2 A| < \infty$ and $B\subset \V\setminus (A \cup \partial^2 A)$.
First, using the definition \eqref{eq:tau_1} of $\tauone,$  since it is a stopping time with respect to the filtration generated by $X_{\partial^2 A}$ and $A$ and $B$ are disjoint from $\partial^2 A$, Lemma \ref{lem:simplicity_PP} ensures that  $X_A(\tauone)=X_A^-(\tauone)$ and $X_B(\tauone)=X_B^-(\tauone).$ In turn, this implies that
\begin{equation}
\label{eq:2-MRf_aux}
    \E[X_A(\tauone)|X_{\partial^2 A}(\tauone),X_B(\tauone)]=\E[X_A^-(\tauone)|X_{\partial^2 A}(\tauone),X_B^-(\tauone)].
\end{equation}
Now, from the definition \eqref{eq:IPS_SDE} of the IPS dynamics, their Markov structure, and the definition \ref{eq:stopped_process} of the stopped process, there exists a measurable function $f_{\partial^2 A}:\stateS^{|\partial^2A|}\times\stateS^{|\partial^2A|}\times\stateS^{|\partial^2A|}\rightarrow\stateS^{|\partial^2A|}$ such that
\begin{equation}
\label{eq:markov_func_repr}
X_{\partial^2 A}(\tauone)=f_{\partial^2 A}(X_{\partial^2 A}^-(\tauone),X_A^-(\tauone),X_B^-(\tauone)).
\end{equation}
By Lemma \ref{lem:marginal_z_2-MRF} and Definition \ref{def:stopped_y_process}, we have that $X_A^-(\tauone)$ and $X_B^-(\tauone)$ are conditionally independent given $X_{\partial^2 A}^-(\tauone)$ and $X_{\partial^2 A}(\tauone)$.
This in turn implies that
\begin{equation*}
\E[X_A^-(\tauone)|X_{\partial^2 A}(\tauone),X_B^-(\tauone)]=\E[X_A^-(\tauone) | X_{\partial^2 A} (\tauone)],
\end{equation*}
which when combined with \eqref{eq:2-MRf_aux} gives
\begin{equation}
\label{eq:2-MRF-tau_1}
      X_A(\tauone)\perp X_B(\tauone) | X_{\partial^2 A}(\tauone).
\end{equation}

For any stopping time $\sigma$, let $X_{\sigma}$ denote the process shifted by $\sigma,$ namely, for any $t>0, X_{\sigma}(t):=X(t+\sigma).$ For $i\geq 1$, let $\sigma_i$ be the time of the $i-$th jump on $\partial^2 A$ (Note that by definition of $\tauone$, see \eqref{eq:tau_1}, $\sigma_1=\tauone$). By Assumption \ref{Ass:A_3}, since the state space is finite, there is only a finite number of jumps on $\partial^2 A$. Therefore, there exists $i_{\infty}\in \N$ such that for all $i>i_{\infty}, \sigma_i=\infty.$ 

Since \eqref{eq:2-MRF-tau_1} holds, $X_{\sigma_1}(0)$ is a 2-MRF, and therefore $X_{\sigma_1}$ satisfies Assumption \ref{Ass:B_1}. We can thus apply Proposition \ref{prop:2MRF_pathwise_stoptime} to the shifted process $X_{\sigma_1}.$ This allows us to prove a time-marginal 2-MRF property for the process shifted by $\sigma_1$, and then proceed iteratively: for any $1\leq i\leq i_{\infty}$, we have
\begin{equation}
\label{eq:2-MRF-sigma_i}
      X_A(t)\perp X_B(t) | X_{\partial^2 A}(t), \sigma_i\leq t\leq \sigma_{i+1}.
\end{equation}
Since $i_{\infty}<\infty,$ the 2-MRF property \eqref{eq:2-MRF-sigma_i} in fact holds for any $t\geq 0$, which concludes the proof.
\end{proof}

\section{Proof of the Markovian projection}
\label{sec:Markov_proj_proof}
This section is devoted to the proof of Theorem \ref{thm:markov_proj}. To that effect, we prove a more general Markovian projection result for pure jump processes, that is, existence and uniqueness of a Markov process with the same time-marginals as the original pure jump process. We state the result in Section \ref{subsec:Markov_proj_theorem} with sufficient conditions for existence and uniqueness of the Markovian projection. In Section \ref{subsec:Markov_proj_Larsson_Long}, we appeal to a result by Larsson and Long \cite{larsson2024markovian} to establish the existence of a Markovian projection as the solution to a martingale problem. In Section \ref{subsec:Markov_proj_proof_uniqueness}, we prove that this martingale problem has a unique solution using its equivalence to an SDE, and that this implies the Feller property for the solution of the martingale problem.
\subsection{A Markovian projection theorem for pure jump processes}
\label{subsec:Markov_proj_theorem}
We aim to establish the existence of a process whose time-marginals coincide with the time-marginals of the solution to a jump SDE with potentially trajectory-dependent jump rates, with the local-field equation \eqref{eq:lf-SDE} in mind as an application.
We first define the trajectory-dependent IPS to which our Markovian projection result will apply.

\begin{definition}[Trajectory-dependent IPS]
\label{def:traj-dep_IPS}
Let $\stateS$ be a finite subset of $\Z$ representing the state space of each particle. Let $\Jmc\subset \stateS- \stateS\setminus\{0\}$ be the set of permissible jump sizes, and let $j_{\max}:=\max \Jmc.$ Fix $d\in\N\setminus\{0\},$ and let $\vertexS:=\{0,\ldots,d-1\}.$
Let $\Nlf=\{\Nlf_v^j\}_{j\in \Jmc, v\in\vertexS}$ be a collection of independent Poisson processes on $\R_+\times\R_+$ with 
unit intensity. For each  $j \in \Jmc$ and $v\in \vertexS$, the rate $\rjmp^j_v:[0,\infty)\times \D^\vertexS \rightarrow [0,\infty)$ is a measurable mapping that represents the size $j$ jump intensity of particle $v,$ with the assumption that if $x,y\in \D^\vertexS$ s.t. $x[t]=y[t],$ $\rjmp^j_v(t,x)=\rjmp^j_v(t,y).$

We define a stochastic process $\Xjmp$ on $\stateS$ as the solution to the following SDE:

    \begin{equation}\label{eq:trajIPS-SDE}
        \Xjmp_v(t)= \Xjmpinitial_v(0)+\sum_{j \in \Jmc}\int_{(0,t)\times \R_+} j  \one_{\{u<\rjmp_v^{j}(s,\Xjmp)\}} \mathbf{\Nlf}^j_v(\diff s,\diff u),\quad v\in \vertexS.
    \end{equation}
\end{definition}

Note that the local-field equation defined in Definition \ref{def:lfe} falls into this framework.

The Markovian projection of a trajectory-dependent IPS can be given as the solution to an SDE. However, due to considerations to existing formulations of Markovian projection results in the literature as well as an anticipation on the proof of our Markovian projection theorem, we note that the Markovian projection can be equivalently defined as the solution to a martingale problem, which we will now introduce. 
Let $\Omega:=\Dmc(\mathbb{R}_+: \mathbb{R}^d)$ and let $\Fmc^0:=\Bmc(\Omega)$ be the corresponding Borel $\sigma$-algebra.
Let $\Xcanon$ be the canonical process, that is,\ $\Xcanon(t,\omega) = \omega(t)$ for $\omega \in \Omega$ and $t \geq 0$. 
Consider the non-local operator $\Gen = (\Gen_t)_{t \geq 0}$ given by
\begin{equation}\label{eq:mart_problem_gen}
\begin{split}
	\Gen_t f(x) \coloneqq \int_{\mathbb{R}^d} \bigl(f(x + \xi) - f(x)  \bigr) \,\k(t,x, \diff \xi),
\end{split}
\end{equation}
for $f \in C^2(\mathbb{R}^d) \cap C_b(\mathbb{R}^d)$ and $x \in \mathbb{R}^d$,
where $\k$ is a L\'evy transition kernel from $\mathbb{R}_+ \times \mathbb{R}^d$ to $\mathbb{R}^d$ such that $\k(t,x,\cdot)$ is supported on a bounded set.
Let $\Fmc^{\Xcanon}$ be the natural filtration generated by $\Xcanon$, and $\Fmb=\{\Fmc_t\}$ be the right-continuous regularization of $\Fmc^{\Xcanon}$.

\begin{definition}[Martingale Problem]
\label{def:larsson_long_mtg_pblm}
Let $\nu_0 \in \mathcal{P}(\mathbb{R}^d)$. A solution to the martingale problem  for $\mathcal{L}$ with initial law $\nu_0$ is a probability measure $\Qmb$ on $(\Omega,\Fmb)$ such that
\begin{enumerate}[label=(\roman*)]
\item $\mathbb{Q} \circ (\Xcanon(0))^{-1} = \nu_0$,

\item for each $f \in C_c^2(\mathbb{R}^d)$, the process
\begin{equation*}
	M^f_t \coloneqq f(\Xcanon_t) - f(\Xcanon(0)) - \int_0^t \Gen_s f(\Xcanon(s))\diff s
\end{equation*}
is well-defined and an $\mathbb{F}$-martingale under $\Qmb$.
\end{enumerate}
\end{definition}
  
We now state our Markovian projection result for history-dependent jump processes:
To apply the results of \cite{larsson2024markovian}, we embed the jump directions $\Jmc^d$ in $\R^d$. Let $\ev$ be the $v^{th}$ standard basis vector in $\R^d$. 

\begin{theorem}
\label{thm:Markov_proj_for_hist_IPS}
Let $\Xjmp$ be a trajectory-dependent IPS as in Definition \ref{def:traj-dep_IPS}. Suppose that there exists $C>0$ such that for every $j\in\Jmc$, every $v\in\vertexS$, and every $t\geq 0,$ the transition rates $\rjmp^j_v$ are c\`agl\`ad and satisfy
\begin{equation}
\label{eq:cond_rates_exist_MP}
\rjmp^j_v(t,\cdot)\leq C.
\end{equation}
Define the rates $\rmjmp_v^{j}:[0,\infty)\times \stateS^\vertexS \rightarrow [0,\infty)$ by
\begin{equation}
    \label{eq:Markov_proj_rates}
    \rmjmp_v^{j}(s,z)=\E[\rjmp^j_v(s,\Xjmp)|\Xjmp(s-)=z].
\end{equation}
Then there exists a unique Feller process $\Xmp$
satisfying  
\begin{equation}
    \Xmp(t)\deq \Xjmp(t) \qquad \forall t\geq 0.
\end{equation}
The process $\Xmp$
 can be defined in one of the following equivalent ways:
\begin{enumerate}
    \item $\Xmp$ is the unique solution to the  SDE
\begin{equation}
\label{eq:Markov_proj_SDE}
     \Xmp_v(t)= \Xjmpinitial_v(0)+\sum_{j \in \Jmc}\int_{(0,t)\times \R_+} j  \one_{\{u<\rmjmp_v^{j}(s,\Xmp(s))\}} \mathbf{\Nlf}^j_v(\diff s,\diff u), v\in \vertexS.
\end{equation}
\item $\Xmp$ is the solution to the martingale problem with $\Gen$ defined as in \eqref{eq:mart_problem_gen} with $\k$ given by
\begin{equation}\label{eq:k_comp}
    \k(t,x,\diff \xi) =\one_{\{ \exists j\in\Jmc,\  v\in\vertexS \ : \ \xi=j \ev\}}\E[\rjmp^j_{v}(s,\Xjmp)) |  \Xjmp(s-)=x]\countm(\diff \xi),
\end{equation}
where $\countm$ is the counting measure.
\end{enumerate}
\end{theorem}
\begin{remark}
Theorem \ref{thm:markov_proj} is an immediate consequence of Theorem \ref{thm:Markov_proj_for_hist_IPS}.
\end{remark}
\subsection{Proof of Theorem \ref{thm:Markov_proj_for_hist_IPS}}
\label{subsec:Markov_proj_for_hist_proof}
\subsubsection{A Markovian projection result from Larsson and Long \cite{larsson2024markovian}}
\label{subsec:Markov_proj_Larsson_Long}

To prove the existence of the Markovian projection in Theorem \ref{thm:Markov_proj_for_hist_IPS}, we will show that we can apply \cite{larsson2024markovian}*{Theorem 3.2} to our setting. Their result is formulated as a sufficient condition for the existence of a solution to the martingale problem for a certain generator such that the time-marginals of this solution and that of the original jump SDE coincide. We start by showing that we can rewrite the generator \eqref{eq:mart_problem_gen} in the following form that will enable us to easily check that the sufficient conditions from \cite{larsson2024markovian} are satisfied in our framework.

Indeed, let $b:\R_+\times\R^d\rightarrow\R^d$ be a measurable map such that for every $f\in C^2(\R^d)\cap C_b(\R^d)$ and every $t\geq 0,$
\begin{equation}
\label{eq:def_b}
b(t,x)\cdot\nabla f(x)=\int_{\R^d}\nabla f(x)\cdot \xi K(t,x,\diff \xi).
\end{equation}
We can then write
\begin{equation}
\label{eq:our_generator}
	\Gen_t f(x) = b(t, x) \cdot \nabla f(x) + \int_{\mathbb{R}^d} \bigl(f(x + \xi) - f(x) - \nabla f(x) \cdot \xi \bm{1}_{\{|\xi| \leq r\}}\bigr) \,\k(t,x, \diff \xi).
\end{equation}
Since the jumps in the trajectory-dependent IPS \eqref{eq:trajIPS-SDE} are bounded, we can choose $r$ large enough so that $\bm{1}_{\{|\xi| \leq r\}}$ is always equal to 1.

\begin{lemma}
\label{lem:MP_existence}
Let $\Xjmp$ be a trajectory-dependent IPS as in Definition \ref{def:traj-dep_IPS}. Suppose that there exists $C>0$ such that for every $j\in\Jmc$, every $v\in\vertexS$, and every $t\geq 0,$ the transition rates $\rjmp^j_v$ are c\`agl\`ad and satisfy
\begin{equation*}
\rjmp^j_v(t,\cdot)\leq C.
\end{equation*}
Let $\k$ be defined as in \eqref{eq:k_comp}.

Then there exists a solution $\Xmp$ to the martingale problem for $\mathcal{L}$, defined in \eqref{eq:mart_problem_gen}, such that for each $t \geq 0$,  $\law(\Xjmp(t))=\law(\Xmp(t))$.
 \end{lemma}

\begin{proof}
This would be a consequence of \cite{larsson2024markovian}*{Theorem 3.2}, with 
\begin{equation}
\label{eq:our_compensator}
\kappa_t(\diff \xi) :=\one_{\{ \exists j\in\Jmc, v\in \vertexS \ : \ \xi=j\ev\}}\rjmp^j_v(t,\Xjmp) \countm(\diff \xi),
\end{equation}
where $\countm$ is the counting measure and
\begin{equation}
\label{eq:our_drift}
\beta_s=\sum_{j\in J}\one_{\{j\in\Jmc, v \ : \ \xi=j\ev\}}\rjmp^j_v(s,x),
\end{equation} 
if we can verify that that the assumptions therein are satisfied. 

To this end, we start by showing that $\int_0^t\kappa_s(\Rd)ds$ has finite expectation for every $t\in\R_+$. Since $\Jmc$ is finite, it follows that
\begin{equation*}
    \text{supp}(\kappa_s)=|\{\xi : \xi= j \ev, j\in\Jmc, v\in\vertexS\}|=d|\Jmc|.
\end{equation*} 
Using \eqref{eq:cond_rates_exist_MP}, it follows that

\begin{align*}
    \E\left[\int_0^t\kappa_s(\Rd)\diff s\right] &= \E\left[ \int_0^t\int_\Rd \one_{\{ \exists j\in\Jmc, v\in\vertexS \ : \ \xi=j\ev\}}\rjmp^v_{j}(s,\Xjmp)) \countm(\diff \xi)\diff s \right]
    \\ &\leq  t |\Jmc|Cd
    \\ &<\infty.
\end{align*}
This establishes \cite{larsson2024markovian}*{condition (3.2)}. 

Now,  since $j_{\max}<\infty$, we have
    \begin{equation}
    \label{eq:cond_bound}
        \int_{\R^d} \frac{|\xi|^2}{1+ |\Xjmp(t)|^2} \kappa_t(d\xi) \leq  j_{\max}^2\int_{\R^d} \kappa_t(d\xi)= j_{\max} |\Jmc|Cd, 
    \end{equation}

Since the jump sizes of $\Xjmp$ are bounded by $j_{\max}$, \eqref{eq:cond_bound} establishes \cite{larsson2024markovian}*{Condition (3.4)}, as argued in \cite{larsson2024markovian}*{Remark 3.5}.
This in turn gives us the existence of a L\'evy transition kernel $\k$ from $\mathbb{R}_+ \times \mathbb{R}^d$ to $\mathbb{R}^d$ such that for Lebesgue-a.e.\ $t \geq 0$,
\begin{equation}\label{eq:mp_condexp}
	\begin{split}
		\k(t, \Xjmp(t-), A)
		&= \mathbb{E}\big[\kappa_t(A) \,\big|\, \Xjmp(t-)\big],\quad
		\forall\, A \in \mathcal{B}(\mathbb{R}^d),
	\end{split}
\end{equation}
as well as the existence of a solution to the martingale problem \eqref{eq:mart_problem_gen} with the desired time-marginals.  Finally, we observe that, since $\countm$ is the counting measure, 
\begin{align*}
    \mathbb{E}\big[\kappa_t(A) \,\big|\, \Xjmp(t-)\big] & = \mathbb{E}\big[\sum_{\xi\in A}\one_{\{ \exists j\in\Jmc, v\in \vertexS \ : \ \xi=j\ev\}}\rjmp^j_v(t,\Xjmp)\big|\, \Xjmp(t-)\big]
    \\ &=\sum_{\xi\in A} \one_{\{ \exists j\in\Jmc, v\in \vertexS \ : \ \xi=j\ev\}}\mathbb{E}\big[\rjmp^j_v(t,\Xjmp)\big|\, \Xjmp(t-)\big],
\end{align*}
which coincides with \eqref{eq:k_comp}.

\end{proof}

\subsubsection{Uniqueness of the Markovian projection and Feller property} 
\label{subsec:Markov_proj_proof_uniqueness}

After establishing existence of the Markovian projection in Section \ref{subsec:Markov_proj_Larsson_Long}, what remains to prove is that the martingale problem satisfied by the Markovian projection of the local-field equations has a unique solution. To do so, we will use the equivalence established by Kurtz \cite{Kurtz2011} between stochastic equations and martingale problems to show that the SDE \eqref{eq:Markov_proj_SDE} is equivalent to the martingale problem associated with \eqref{eq:our_generator}. 

Using \eqref{eq:our_drift}, we can rewrite \eqref{eq:Markov_proj_SDE} as 

\begin{equation}
\label{eq:SDE_Kurtz_form}
\begin{split}
\Xmp_v(t)=&\Xmp_v(0)+\int_0^t \beta_s(\Xmp_v(s))\diff s+\int_{[0,1]\times \Jmc\times [0,t]}\frac{j}{2}\one_{[0,\rmjmp_v^j(s,\Xmp(s))]}(u)\bar{\mathbf{N}}^j_v(\diff u\times \diff s)\countm(\diff j)\\
&+\int_{[0,1]\times \Jmc\times [0,t]}\frac{j}{2}\one_{[0,\rmjmp_v^j(s,\Xmp(s))]}(u)\mathbf{N}^j_v(\diff u\times \diff s)\countm(\diff j),        
\end{split}
\end{equation}
where $(\mathbf{N}^j_v)_{v\in V, j\in\Jmc}$ are independent Poisson point processes on $\R_+\times\R_+$ with unit intensity, $\bar{\mathbf{N}}^j_v(\diff s\times \diff u):=\mathbf{N}^j_v(\diff s \times \diff u)-\diff s\diff u$, is the compensated version of $\mathbf{N}^j_v$, and $\beta_s$ is as defined in \eqref{eq:our_drift}.

We state here a theorem of Kurtz \cite{Kurtz2011}[Theorem 2.3] applied to our setting in order to introduce the SDE equivalent to our martingale problem. In \cite{Kurtz2011}*{Section 2}, the result is written for time-homogeneous Markovian martingale problems, but extends to the time-inhomogeneous case. 

\begin{proposition}
\label{prop:kurtz_mp_sde_equiv}  
Assume that for any compact $B \in \R^d,$
\begin{equation}
\label{eq:kurtz_thm:assm}
\sup_{x\in B} (|\beta_s(x)|+\int_\Jmc \rmjmp^j_v(s,\Xmp(s^-))j^2\countm(\diff j)+\int_\Jmc \rmjmp^j_v(s,\Xmp(s^-))\countm(\diff j))< \infty,
\end{equation}
and that for $f\in C^2_c(\R^d), \mathcal{L}_t f \in B(\R^d).$
Then any solution to the martingale problem associated with \eqref{eq:our_generator} is a weak solution of \eqref{eq:SDE_Kurtz_form}.
\end{proposition}

Note that since the jump rates $\rmjmp_v^j$ are bounded, condition \eqref{eq:kurtz_thm:assm} is satisfied. Since the SDE \eqref{eq:Markov_proj_SDE} is well-posed from considerations equivalent to those in Proposition \ref{prop:MLFE-wellposed}, we obtain the desired uniqueness of the solution to the martingale problem. Now, all that remains is to show that this solution is indeed a Markov process. To that end, we state a result on the fact that uniqueness in law of the solution of a state-dependent SDE implies that the solution is a Feller process, and therefore is in particular a Markov process.
\begin{proposition}\label{prop:uniq_imply_feller}
    Suppose that the SDE \eqref{eq:Markov_proj_SDE} has a unique-in-law solution $\Xmp$. Then $\Xmp$ is a Feller process. 
\end{proposition}
\begin{proof} We establish this using \cite{xi2019jump}*{Proposition 4.2}, which considers SDEs of the form
\begin{equation}\label{eq:xi-eq}
    \diff X (t)= b(X(t))\diff t+\sigma(X(t))\diff W(t) +\int_U c(X(t-),u)\tilde{N}( \diff t, \diff u),
\end{equation}
    where $(U,\Umc,\mu)$ is a sigma-finite measure space, $b : \Rd \rightarrow \Rd$ , $\sigma : \Rd \rightarrow \R^{d\times d}$ and $c : \Rd \times U \rightarrow\Rd$ are Borel measurable functions, $W$ is a $d$-dimensional Brownian motion, and $\tilde{N}$ is a compensated Poisson process on $\R\times U$ with intensity $\diff t \otimes \mu(u)$. To relate \eqref{eq:xi-eq} to  \eqref{eq:Markov_proj_SDE}, we set $\sigma\equiv0$, $U=\R_+\times \Jmc^d$, $\mu= \textrm{Leb}(\diff t) \times \countm$ (where we recall $\countm$ is the counting measure on the discrete space $\Umc$ and $\textrm{Leb}$ is the Lebesgue measure). Moreover, we have
    \begin{equation*}
        c(x,u)=c(x,(r,\xi))= \one_{\{\exists j\in\Jmc, v\in G :\ \xi=j\ev\}} j \one_{\{r<\rmjmp_v^j(t,x)\}},
    \end{equation*}
    and 
    \begin{equation*}
    b(x)=\int_U c(x,u) \mu(du).
    \end{equation*}
    In this regime, the uniqueness in law of the solution to \eqref{eq:Markov_proj_SDE} is equivalent to \cite{xi2019jump}*{Assumption 4.1}. Assumption 2.3 in \cite{xi2019jump} is satisfied with $\rho(r)=r$ and $\delta_0=0.5$ therein.
\end{proof}

\section{Proof of well-posedeness of the ODE} 
\label{sec:ode_well-posed}
\draftnote{This section is new as of Feb 2.}In this section we establish the proof of Proposition \ref{prop:ODE-wellposed} which establishes the existence and uniqueness of a solution to the ODE system \eqref{eq:MF-ODE-UGW}-\eqref{eq:MF-ODE-UGW_initial}. 

\begin{proof}[Proof of Proposition \ref{prop:ODE-wellposed}] We show that the right-hand side of \eqref{eq:MF-ODE-UGW} is Lipshitz in $\lawMF_t\in \Pmc^\offspring$, from which the assertion of the Proposition follows by invoking Picard's existence and uniqueness theorem.

Since the sum on right-hand side of \eqref{eq:MF-ODE} is over finitely many elements, it is enough to consider the functions $\Pmc^\offspring\ni q \mapsto q(\vec{a})\Psibar_{a_0,a_1}(q)$, for $\vec{a}\in C^\offspring$,
where 
\begin{equation}\label{eq:psibar}
    \Psibar_{a_0,a_1}(q):= \Psibar_{a_0,a_1}^{t,j}(q):= \frac{\sum_{\vec{b}\in C^\offspring} k(\vec{b})\rr^j(t,\vec{b}) q(\vec{b}) \one_{\{ b_\root=a_1,b_1=a_0\}}}{\sum_{\vec{c}\in C^\offspring} k(\vec{c}) q(\vec{c}) \one_{\{ c_\root=a_1,c_1=a_0\}}},
\end{equation}
for fixed $t\in\R_+$ and $j\in\Jmc$, where $\rr$ and $k(\vec{b})$ are defined in \eqref{def:deparametrize_r} and \eqref{eq:degree_of_vec} respectively.

We proceed by showing that the partial derivatives $\partial_{q(\vec{b})} q(\vec{a})\Psibar_{a_0,a_1}(q)$ are bounded. To simplify notation, for fixed $\vec{a}\in C^\offspring$ and $j\in\Jmc$, we let $N=N_{a_0,a_{1}}^j(q)$ and $D=D_{a_0,a_{1}}^j(q)$ be the numerator and denominator on the right-hand side of \eqref{eq:psibar}.
By Assumptions \ref{Ass:A_1} and \ref{Ass:A_2}, for all $\vec{c}\in C^\offspring$, $k(\vec{c})\leq d_{\max}<\infty$, where $\dmax=\max\{k : \offspring(k)>0\}$, and there exists $R\in(0,\infty)$ such that $\rr^j(t,\vec{c})<R$ for all $t\in\R_+$ and $j\in\Jmc$. It follows that 
\begin{equation}\label{eq:N_bound}
    N\leq \dmax R
\end{equation}
and
\begin{equation}\label{eq:D_bound}
    D\leq \dmax.
\end{equation}

When $\vec{b}\neq \vec{a}$, using the quotient rule we have that
\begin{equation}\label{eq:partial_d}
    \partial_{q(\vec{b})} q(\vec{a})\Psibar_{a_0,a_1}(q)=q(\vec{a})\frac{k(\vec{b})\one_{\{ b_\root=a_1,b_1=a_0\}} (\rr^j(r,\vec{b})D - N)}{D^2}.
\end{equation}

As the numerator in \eqref{eq:partial_d} is bounded, it is sufficient to bound $|\partial_{q(\vec{b})} q(\vec{a})\Psibar_{a_0,a_1}(q)|$ for $q$ such that $D$ is near $0$. 
Let 
\begin{equation*}
    \delta(q)=q(\{x\in C^\offspring : x_0=a_0, x_1=a_1\})= \sum_{\vec{c}\in C^\offspring} \one_{\{c_0=a_0, c_1=a_1\}}q(\vec{c}).
\end{equation*}
 Since every term in the sum is non-negative,  we have that for every $\vec{c}\in\Pmc^\offspring$ with $c_0=a_0$ and $c_1=a+1$,
\begin{equation*}
    q(\vec{c})\leq \delta(q).
\end{equation*}
In particular, we have $q(\vec{a})\leq \delta(q)$.
Furthermore, using \eqref{eq:psibar} and the bounds on $k(\vec{b})$, $\rr$, $N$ and $D$, we have that
\begin{equation*}
\delta(q)\leq D\leq d_{\max}\delta(q)   
\end{equation*}
and
\begin{equation*}
 N\leq d_{\max} R \delta(q).   
\end{equation*}
It follows that 
    
\begin{equation}\label{eq:partial_d_1}
    |\partial_{q(\vec{b})} q(\vec{a})\Psibar_{a_0,a_1}(q)| \leq \delta(q) \frac{2\delta(q) \dmax^2 R }{ (\delta(q))^2}   \leq 2 \dmax^2 R.
\end{equation}

Now, consider the case $\vec{b}=\vec{a}$. Then, by the product rule, we have

\begin{align*}
    \partial_{q(\vec{a})} q(\vec{a})\Psibar_{a_0,a_1}(q)&= \Psibar_{a_0,a_1}(q) \partial_{q(\vec{a})} q(\vec{a})+ q(\vec{a}) \partial_{q(\vec{a})} \Psibar_{a_0,a_1}(q)
    \\&= \Psibar_{a_0,a_1}(q) + q(\vec{a}) \partial_{q(\vec{a})} \Psibar_{a_0,a_1}(q).
\end{align*}
The second term is bounded as in \eqref{eq:partial_d_1}. The first term can be bounded as follows:
\begin{align*}
   | \Psibar_{a_0,a_1}(q)|=\left|\frac{N}{D}\right|\leq\frac{\dmax R \delta(q)}{\delta(q)}= \dmax R.
\end{align*}

Therefore, the partial derivatives  $\partial_{q(\vec{b})} q(\vec{a})\Psibar_{a_0,a_1}(q)$ are uniformly bounded on $C^\offspring$ and it follows that $q\rightarrow q(\vec{a})\Psibar^{j,t}_{a_0,a_1}(q)$ is globally Lipshitz for all $\vec{a}\in C^\offspring$, $t\in\R_+$ and $j\in \Jmc$. Existence and uniqueness of the solution to \eqref{eq:MF-ODE-UGW}-\eqref{eq:MF-ODE-UGW_initial} follows from Picard's theorem.

\end{proof}
\begin{funding}
The authors were supported by the Office of Naval Research under the Vannevar Bush Faculty Fellowship N0014-21-1-2887.
\end{funding}

\bibliography{biblio_draft}

\end{document}